\documentclass[12pt,reqno]{amsart}

\usepackage{amsfonts,amsmath,amsthm}
\usepackage{amssymb,epsfig}
\usepackage{enumerate} 

\usepackage{color} 
\definecolor{vert}{rgb}{0,0.6,0}

\usepackage[text={425pt,650pt},centering]{geometry}
\usepackage{graphicx}
\usepackage{epsfig}
\usepackage{tikz,esvect}
\usetikzlibrary{3d,calc,shapes}

\usepackage{caption}
\usepackage{color} 
\definecolor{vert}{rgb}{0,0.6,0}

\numberwithin{figure}{section}

\theoremstyle{plain}
\newtheorem{thm}{Theorem}[section]

\newtheorem{defn}{Definition}

\newtheorem{lem}[thm]{Lemma}
\newtheorem{cor}[thm]{Corollary}
\newtheorem{prop}[thm]{Proposition}
\theoremstyle{remark}
\newtheorem{rem}{\bf{Remark}}
\numberwithin{equation}{section}

\newcommand{\N}{\mathbb{N}}

\newcommand{\R}{\mathbb{R}}

\newcommand{\T}{\mathbb{T}}
\newcommand{\Z}{\mathbb{Z}}

\newcommand{\cL}{\mathcal{L}}
\newcommand{\cM}{\mathcal{M}}

\newcommand{\AC}{{\rm AC\,}}

\newcommand{\BUC}{{\rm BUC\,}}

\newcommand{\Lip}{{\rm Lip\,}}

\newcommand{\al}{\alpha}
\newcommand{\gam}{\gamma}
\newcommand{\del}{\delta}
\newcommand{\ep}{\varepsilon}

\newcommand{\lam}{\lambda}

\newcommand{\ol}{\overline}

\begin{document}

\title[Rate of convergence]
{Rate of convergence in periodic homogenization of Hamilton-Jacobi equations: the convex setting}

\author[H. MITAKE, H. V. TRAN, Y. YU]
{Hiroyoshi Mitake, Hung V. Tran, Yifeng Yu}

\thanks{
The work of HM was partially supported by the JSPS grants: KAKENHI \#15K17574, \#26287024,  \#16H03948.
The work of HT is partially supported by NSF grant DMS-1664424.
The work of YY is partially supported by NSF CAREER award \#1151919.
}

\address[H. Mitake]{
Graduate School of Mathematical Sciences, 
University of Tokyo 
3-8-1 Komaba, Meguro-ku, Tokyo, 153-8914, Japan}
\email{mitake@ms.u-tokyo.ac.jp}

\address[H. V. Tran]
{
Department of Mathematics, 
University of Wisconsin Madison, Van Vleck Hall, 480 Lincoln Drive, Madison, Wisconsin 53706, USA}
\email{hung@math.wisc.edu}

\address[Y. Yu]
{
Department of Mathematics, 
University of California at Irvine, 
California 92697, USA}
\email{yyu1@math.uci.edu}


\maketitle

\begin{abstract}
We study the rate of convergence of $u^\ep$, as $\ep \to 0+$, to $u$ in periodic homogenization of Hamilton-Jacobi equations.
Here, $u^{\ep}$ and $u$ are viscosity solutions to the oscillatory Hamilton-Jacobi equation and its effective equation 
\begin{equation*}
{\rm (C)_\ep} \qquad 
\begin{cases}
 u_t^\ep+H\left(\frac{x}{\ep},Du^\ep\right)=0 \qquad &\text{in} \ \R^n \times (0,\infty),\\
u^\ep(x,0)=g(x) \qquad &\text{on} \ \R^n,
\end{cases} 
\end{equation*}
and
\begin{equation*}
{\rm (C)} \qquad 
\begin{cases}
 u_t+\ol{H}\left(Du\right)=0 \qquad &\text{in} \ \R^n \times (0,\infty),\\
u(x,0)=g(x) \qquad &\text{on} \ \R^n,
\end{cases} 
\end{equation*}
respectively. 
We assume  that the Hamiltonian $H=H(y,p)$ is coercive and convex in the $p$ variable and is $\Z^n$-periodic in the $y$ variable,
and the initial data $g$ is bounded and Lipschitz continuous.
Here, $\ol{H}$ is the effective Hamiltonian.

\smallskip

We prove that

\noindent(i)
\[
u^\ep(x,t)\geq u(x,t)- C\ep \quad \text{ for all $(x,t)\in \R^n\times [0,\infty)$},
\]
where $C$ depends only on $H$ and $\|Dg\|_{L^\infty(\R^n)}$;
\smallskip

\noindent(ii)  For fixed $(x,t) \in \R^n \times (0,\infty)$, if $u$ is differentiable at $(x,t)$ 
and $\overline H$ is twice differentiable at $p=Du(x,t)$, then
\[
u^\ep(x,t)\leq u(x,t)+\widetilde C_p t\ep + C\ep
\]
provided that $g\in C^2(\R^n)$ with $\|g\|_{C^2(\R^n)}<\infty$.
The constant $\widetilde C_p$ depends only on $H, \ol{H}, p$ and $g$.   
If $g$ is only Lipschitz continuous, then  the above inequality in (ii)  is changed into $u^\ep(x,t)\leq u(x,t)+C_p \sqrt{t\ep} + C\ep$.

\smallskip

When $n=2$ and $H$ is positively homogeneous in $p$ of some fixed degree $k \geq 1$, utilizing the Aubry-Mather theory, we obtain the optimal convergence rate $O(\ep)$
\[
|u^{\ep}(x,t)-u(x,t) |\leq C \ep \quad \text{ for all $(x,t)\in \R^2 \times [0,\infty)$.}
\]
 Here $C$ depends only on $H$ and $\|Dg\|_{L^\infty(\R^2)}$.

\smallskip

When $n=1$,  the optimal convergence rate $O(\ep)$ is established  for any coercive and convex $H$.

\smallskip

The convergence rate turns out to have deep connections with the dynamics of the  underlying Hamiltonian system and the shape of the effective Hamiltonian $\overline H$. Some related results and counter-examples  are obtained as well.
\end{abstract}


\section{Introduction}

We first give a brief description of the periodic homogenization theory for Hamilton-Jacobi equations.
For each $\ep>0$, let $u^\ep \in C(\R^n \times [0,\infty))$ be the viscosity solution to (C)$_\ep$.
Here, the Hamiltonian $H=H(y,p):\R^n \times \R^n \to \R$ is a  given continuous function, which satisfies
\begin{itemize}
\item[(H1)] For each $p\in \R^n$, $y \mapsto H(y,p)$ is $\Z^n$-periodic.

\item[(H2)] $H$ is coercive in $p$, that is, uniformly for $y\in \T^n:=\R^n/\Z^n$,
\[
\lim_{|p| \to \infty} H(y,p) = +\infty.
\]
\end{itemize}
The initial data $g\in \BUC(\R^n)$, the set of bounded, uniformly continuous functions on $\R^n$.

It was proved in \cite{LPV} that, under assumptions (H1)--(H2), $u^\ep$ converges to $u$ locally uniformly on $\R^n \times [0,\infty)$ as $\ep \to 0$,
and $u$ solves the effective equation (C).
The effective Hamiltonian $\ol{H} \in C(\R^n)$ is determined in a nonlinear way by $H$ through the cell problem as following.  
For each $p \in \R^n$, it was shown in \cite{LPV} that there exists a unique constant $\ol H(p)\in \R$ such that the following cell (ergodic) problem has a continuous viscosity solution
\begin{equation*}
 {\rm (E)}_p \qquad H(y,p+Dv)=\ol H(p)  \qquad \text{in} \ \T^n.
\end{equation*}
If needed, we write $v=v(y,p)$ or $v=v_p(y)$ to demonstrate clearly the nonlinear dependence of $v$ on $p$. It is worth mentioning that in general $v(y,p)$ is not unique even up to additive constants.

Our main goal in this paper is to obtain rate of convergence of $u^\ep$ to $u$, that is, an optimal bound for $\|u^\ep-u\|_{L^\infty(\R^n \times [0,T])}$ for any given $T>0$ as $\ep \to 0+$.  
Heuristically, owing to the two-scale asymptotic expansion, 
\begin{equation}\label{expansion}
u^{\ep}(x,t) \approx u(x,t)+\ep v\left({x\over \ep}, Du(x,t)\right)+O(\ep^2),
\end{equation}
and hence, convergence rate looks like $O(\ep)$ which, if achievable, is clearly optimal.  However, it is basically hopeless to justify this expansion due to two reasons:

\begin{itemize}
\item[(1)] in general, there does not exist a continuous selection of $v(y,p)$ (or even supersolutions) with respect to $p$, let alone Lipschitz continuous selection; 
\item[(2)] the solution $u(x,t)$ to (C) is only Lipschitz in $(x,t)$, and is usually not $C^1$, let alone $C^2$. 
\end{itemize}

\noindent Up to now, the best known convergence rate is $O(\ep^{1/3})$ obtained in \cite{CDI} 
through a kind of weak formulation of (\ref{expansion}) using the ergodic problems and their discounted approximations.  
To be more precise, \cite{CDI} deals with  stationary problems, but the approach can be easily adjusted to handle the Cauchy problem.
Point (1) above poses an essential obstacle to improve the convergence rate by modifying  the method in \cite{CDI}. 
See Section \ref{sec:non-cont} for details. 
See also \cite{ACS, CCM, LYZ, MT}  and the references therein for related results on rate of convergence of first order Hamilton-Jacobi equations in various settings.

\subsection{Main results}
We now describe our main results. In this paper, we assume further that
\begin{itemize}
\item[(H3)] For each $y\in \T^n$, $p \mapsto H(y,p)$ is convex.
\end{itemize}

\subsubsection*{Multi dimensions}
For $n\geq 1$, denote 
\[
S=\left\{p\in \R^n\,:\,  \text{$\ol{H}$ is twice differentiable at $p$}\right\}.
\] 
\begin{thm} \label{thm:nd}
Assume {\rm (H1)--(H3)} and $g \in \BUC(\R^n) \cap \Lip(\R^n)$.
For $\ep>0$, let $u^\ep$ be the viscosity solution to {\rm (C)$_\ep$}.
Let $u$ be the viscosity solution to {\rm (C)}.
Then,

{\rm (i)} \begin{equation}\label{lower-bound}
u^\ep (x,t)\geq u(x,t)- C  \ep \qquad \text{for  all $(x,t)\in \R^n\times [0,\infty)$}.
\end{equation}
The constant $C>0$ in {\rm (i)} and {\rm (ii)} below depends only on $H$ and $\|Dg\|_{L^\infty(\R^n)}$. 

{\rm (ii)} For fixed $(x,t)\in \R^n\times (0, \infty)$,  if  $u$ is differentiable at $(x,t)$ and  $p=Du(x,t)\in  S$, then
\begin{equation}\label{upper-bound}
u^{\ep}(x,t)\leq u(x,t)+C_p  \sqrt {t\ep} + C\ep.
\end{equation}
Here $C_p>0$ is a constant depending on $H, \ol{H},p$ and $\|Dg\|_{L^\infty(\R^n)}$. 
If we further assume that the initial data $g\in C^2(\R^n)$ with $\|g\|_{C^2(\R^n)}<\infty$, then 
\begin{equation}\label{upper-bound-C2}
u^{\ep}(x,t)\leq u(x,t)+\widetilde C_{p} t\ep + C\ep.
\end{equation}
Here $\widetilde C_{p} $ is a constant depending on  $H, \ol{H},p$ and  $\|g\|_{C^2(\R^n)}$.
\end{thm}

\begin{rem} In general, it is probably impossible to tell whether $\ol H$ is twice differentiable at a given $p$. Nevertheless, under assumptions (H1)--(H3), we have that $\ol{H}$ is convex and coercive. The Alexandrov theorem yields that $\R^n\backslash S$ has zero Lebesgue measure. 
Thus, Theorem \ref{thm:nd} sort of  implies that for`` typical"   $g\in C^2$ and Hamiltonian $H$, 
the convergence rate is  $O(\ep)$ for ``most" $(x,t)\in \R^n \times (0,T]$ for fixed $T>0$, which is optimal.
Besides the shape of $\ol H$, the upper bound of $u^{\ep}-u$  is also closely related to  a challenging issue in weak KAM theory (or Aubry-Mather theory),
which is to identify  how fast the average slope of a backward characteristic defined in Definition \ref{def:back} (or orbits in the Mather set)
\[
{\xi(t)\over t}
\]
tends to $D\overline H(p)$ as $t\to -\infty$ for $p\in  \R^n$ and $n>1$.
 See Lemma \ref{lem:rate-connection} for details. When $n\geq 3$, except for some special cases, the above statistical type upper bound \eqref{upper-bound} and \eqref{upper-bound-C2} might be the best we can obtain since very little is known about the shape of $\ol H$ and the structure  of Aubry and Mather sets. See Remark \ref{rem:ndgeneral} for more comments.  The classical Hedlund example \cite{Ba1989,Hed} might be the only well-understood and interesting   three dimensional example, where the convergence rate turns out to be $O(\ep)$ (Theorem \ref{Hedlund-optimal}). 

It is reasonable to speculate that when $n\geq 3$, in general, the optimal convergence rate $O(\ep)$ cannot be achieved for every point $(x,t)$. However, we have not been able to find an explicit example with a fractional convergence rate, which basically involves determining  the long time behavior of a trajectory in a chaotic Hamiltonian system. 
\end{rem}

\begin{rem}If the initial data $g \in \Lip(\R^n)$ is concave, 
it is quite easy to get the optimal convergence rate $O(\ep)$ by using a simple PDE approach based on convexity of $H$. 
See Theorem \ref{thm:concave}.
\end{rem}

\subsubsection*{Two dimensions}
 
 In two dimensions ($n=2$),  orbits in the Mather set can be identified with circle  homeomorphisms (see Appendix), which provide nice control of the convergence of the above average slope.   
 Employing this  special structure and adding a positive homogeneity assumption on $H$,  we are able to establish the optimal rate of convergence.

\begin{thm} \label{thm:2d}
Assume  $n=2$, {\rm (H1)--(H3)} and $g \in \BUC(\R^2) \cap \Lip(\R^2)$. 
Assume further that $H$ is positively homogeneous of degree $k$ in $p$ for some $k \geq 1$, that is, $H(y,\lam p) = \lam^k H(y,p)$ for all $(\lam,y,p) \in [0,\infty) \times \T^2 \times \R^2$. 
Then,
\begin{equation}\label{bound-2d}
|u^\ep (x,t)-u(x,t)|\leq  C  \ep \qquad \text{for  all $(x,t)\in \R^2\times [0,\infty)$}.
\end{equation}
Here $C>0$ is a constant depending only on $H$ and $\|Dg\|_{L^\infty(\R^2)}$. 
\end{thm}

Note that $k=1$ corresponds to Hamiltonians associated with  the  front propagation,  which is  probably one of the most physically relevant situations in the homogenization theory.  See Remark \ref{rem:mechanical-Ham} for comments about other types of Hamiltonians, e.g., $H(y,p)={1\over 2}|p|^2+V(y)$.

\subsubsection*{One dimension}
In one dimensional setting, we also obtain optimal  convergence rate $O(\ep)$ for general convex and coercive $H$ through the following theorem.
\begin{thm} \label{thm:1d}
Assume that $n=1$ and {\rm (H1)--(H3)} hold.
Assume further that $g \in \Lip(\R) \cap \BUC(\R)$. 
Then,
\begin{equation}\label{rate-class-1d}
\|u^\ep - u\|_{L^\infty(\R \times [0,\infty))} \leq C  \ep.
\end{equation}
Here $C$ is a constant depending only on $H$ and $\|g'\|_{L^\infty(\R)}$.
\end{thm}

It remains a very interesting problem to determine the optimal convergence rate in the nonconvex setting when $n=1$. 
We conjecture that  the optimal convergence rate is $\sqrt {\ep}$ for general coercive $H$. 

\medskip

We would like to emphasize that the constant $C$ in our results does not depend on the smoothness of $H$, which will be made clear in proofs.

\subsubsection*{Lack of continuous selection of $v(y,p)$ with respect to $p$}
This  fact should be known to experts although we could not find any explicit example in the literature.  
In Section \ref{sec:non-cont}, we give an example in three dimensions where there is no continuous selection of $v(y,p)$ with respect to $p$.
See Theorem \ref{thm:non-cont}. This example resembles some features of the more sophisticated Hedlund example \cite{Hed} (see section 2.4).

\subsection*{Organization of the paper}
The paper is organized as follows.
We first provide settings and simplifications that can be used throughout the paper at the beginning of Section \ref{sec:proof-main}.
The latter part of Section \ref{sec:proof-main} is devoted to the proof of Theorem \ref{thm:nd}.
We give two further situations that  upper bound \eqref{upper-bound} can be improved in Corollary \ref{cor:flat}  and Theorem \ref{thm:concave}.
The classical Hedlund example is analyzed at the end of this section (Theorem \ref{Hedlund-optimal}).
The proof of Theorem \ref{thm:2d} for positively homogeneous Hamiltonians  in two dimensions is shown in Section \ref{sec:2d}.
In Section \ref{sec:1d}, we give the proof of Theorem \ref{thm:1d}.
In Subsection \ref{subsec:ex}, we give a simple example to demonstrate that $O(\ep)$ is indeed the optimal rate of convergence in one dimension.
We discuss the nonexistence of continuous  selection $v(y,p)$ with respect to $p$ in Section \ref{sec:non-cont}. Our approaches combine tools from PDE, weak KAM theory and the Aubry-Mather theory. Some basic backgrounds are provided in Appendix (Section \ref{appen}).

\subsection*{Acknowledgment}  
We are deeply thankful to Hitoshi Ishii, who provides us invaluable comments and suggestions, 
which help much in vastly improving the presentation of the paper.  
We also would like to thank Weinan E and Jinxin Xue for helpful comments and discussions.


\section{Settings, Simplifications and Proof of Theorem \ref{thm:nd}} \label{sec:proof-main}
In this section, (H1)--(H3) are always in force.
We always assume that $g \in \BUC(\R^n) \cap \Lip(\R^n)$.

\subsection{Settings and Simplifications} 
By the comparison principle, it is straightforward that 
\[
\|Du^{\ep}(x,t)\|_{L^{\infty}(\R^n)}\leq C_0.
\]
Here $C_0>0$ is a constant depending only on $H$ and $\|Dg\|_{L^\infty(\R^n)}$.  
Accordingly,  values of $H(y,p)$ for $|p|>C_0$ are irrelevant.  Hence, without loss of generality, we may further assume that $H$ grows quadratically, that is,
\begin{equation}\label{quad-growth}
 {1\over 2}|p|^2-K_0\leq H(y,p)\leq {1\over 2} |p|^2+K_0 \quad \text{ for all } (y,p) \in \T^n \times \R^n,
\end{equation}
for some $K_0 >1$.  
Then, we also have that
\begin{equation}\label{quad-growth-Hbar}
 {1\over 2}|p|^2-K_0\leq \ol{H}(p) \leq {1\over 2} |p|^2+K_0 \quad \text{ for all } p \in \R^n.
\end{equation}
We use \eqref{quad-growth} and  \eqref{quad-growth-Hbar} to get that, for each $v_p \in \Lip(\T^n)$ solving (E)$_p$,
\begin{equation*}
\|Dv_p\|_{L^\infty(\T^n)} \leq 2(|p|+K_0).
\end{equation*}
In particular,
\begin{equation}\label{osc-v-p}
\max_{\T^n} v_p - \min_{\T^n} v_p \leq 2 \sqrt{n} (|p|+K_0).
\end{equation}

Let $L(y,q)$ and $\overline L(q)$ be the Lagrangians (Legendre transforms) of the Hamiltonians $H(y,p)$ and $\overline H(p)$, respectively. 
It is clear that
\begin{equation}\label{quad-growth-L}
 {1\over 2}|q|^2-K_0\leq L(y,q)\leq {1\over 2} |q|^2+K_0 \quad \text{ for all } (y,q) \in \T^n \times \R^n,
\end{equation}
and
\begin{equation*}\label{quad-growth-Lbar}
 {1\over 2}|q|^2-K_0\leq \ol{L}(q) \leq {1\over 2} |q|^2+K_0 \quad \text{ for all } q \in \R^n.
\end{equation*}

\smallskip

For $(x,t)\in \R^n \times (0,\infty)$, the optimal control formula for the solution to (C)$_\ep$ implies
\begin{equation}\label{oc}
u^{\ep}(x,t)=\inf_{\substack{ \ep \eta(0)=x \\ \eta \in \AC([-\ep^{-1}t,0])}}\left\{g\left(\ep\eta \left(-\ep^{-1}t\right)\right)+\ep\int_{-\ep^{-1}t}^{0}L(\eta(s),\dot \eta(s))\,ds\right\}. 
\end{equation}
Here, $ \AC([-\ep^{-1}t,0])$ is the set of all absolutely continuous functions on $[-\ep^{-1}t,0]$ with values in $\R^n$.

\subsection{Proof of Theorem \ref{thm:nd}}

We divide the proof into two parts.
We first derive the lower bound \eqref{lower-bound}.  

\begin{proof}[Proof of  lower bound \eqref{lower-bound}]
Let us mention right away that to get \eqref{lower-bound}, we only need $g \in \BUC(\R^n) \cap \Lip(\R^n)$. Throughout this proof,  $C, C_1$ denote  constants depending only on $\|Dg\|_{L^\infty(\R^n)}$ and $K_0$, which might change from line to line.  

By scaling and translation, it suffices to prove that \eqref{lower-bound} holds for $(x,t)=(0,1)$.
In other words, we aim to show
\begin{equation}\label{3-class-1}
u^\ep(0,1) - u(0,1)  \geq -C  \ep.
\end{equation}

Without loss of generality, we may assume that $g(0)=0$ by considering $\tilde g=g(x)-g(0)$. Hence
\begin{equation}\label{g-bound}
|g(x)|\leq C|x| \quad \text{ for all } x\in \R^n.
\end{equation}
The optimal control formula \eqref{oc} gives us that
\[
u^{\ep}(0,1)=\inf_{\substack{  \eta(0)=0 \\ \eta \in \AC([-\ep^{-1},0])}}\left\{g\left(\ep\eta \left(-\ep^{-1}\right)\right)+\ep\int_{-\ep^{-1}}^{0}L(\eta(t), \dot \eta(t))\,dt\right\}.
\]
Due to \eqref{quad-growth-L} and Jensen's inequality,
\[
\ep\int_{-\ep^{-1}}^{0}L(\eta(t), \dot \eta(t))\,dt \geq  \ep\int_{-\ep^{-1}}^{0}\left(\frac{|\dot \eta(t)|^2}{2} - K_0 \right)\,dt \geq  {1\over 2}\ep^2\left|\eta (-\ep^{-1})\right|^2-K_0.
\]
Owing to \eqref{g-bound},  there exists $C>0$ such that 
\begin{equation}\label{LH-new}
u^{\ep}(0,1)=\inf_{\substack{\eta(0)=0, \\ \ep\left|\eta (-\ep^{-1})\right|\leq C}}\left\{g\left(\ep\eta \left(-\ep^{-1}\right)\right)+\ep\int_{-\ep^{-1}}^{0}L(\eta(t),\dot \eta(t))\,dt\right\}.
\end{equation}
Clearly, there exists $C_1>0$ such that for any $|q|\leq C$, 
\begin{equation}\label{Lbar-range}
\ol L(q)=\sup_{p\in  \R^n}\left\{p\cdot q-\ol H(p)\right\}=\sup_{|p|\leq C_1}\{p\cdot q-\ol H(p)\}.
\end{equation}
For  $p\in  \R^n$, let $v_p \in \Lip(\T^n)$ be  a viscosity solution to (E)$_p$ such that $v_p(0)=0$.  
Then for any Lipschitz continuous curve $\eta:[- \ep^{-1},0]\to \R^n$,  
\[
\int_{-\ep^{-1}}^{0}L(\eta(t),\dot \eta(t))+\overline H(p)\,dt\geq p\cdot \eta(0)-p\cdot \eta\left(-\ep^{-1}\right)+v_p(\eta(0))-v_p\left(\eta\left(-\ep^{-1}\right)\right).
\]
The above fact is easily seen in case $v_p$ is smooth as
\begin{align*}
&\int_{-\ep^{-1}}^{0}L(\eta(t),\dot \eta(t))+\overline H(p)\,dt = \int_{-\ep^{-1}}^{0}L(\eta(t),\dot \eta(t))+H(\eta(t),p+Dv_p(\eta(t)))\,dt\\
\geq & \int_{-\ep^{-1}}^{0}\dot \eta(t)\cdot (p+Dv_p(\eta(t)))\,dt = p\cdot \eta(0)-p\cdot \eta\left(-\ep^{-1}\right)+v_p(\eta(0))-v_p\left(\eta\left(-\ep^{-1}\right)\right).
\end{align*}
Otherwise, it can be obtained by using standard convolutions of $v_p$.
See \cite{Fa} for example.
Accordingly,  for $\eta(0)=0$ and $\ep\left|\eta (-\ep^{-1})\right|\leq C$, in light of  \eqref{osc-v-p} and \eqref{Lbar-range},
\begin{align*}
\ep\int_{-\ep^{-1}}^{0}(L(\eta(t),\dot \eta(t))\,dt&\geq \sup_{p\in \R^n}\left\{p\cdot \left(-\ep\eta\left(-\ep^{-1}\right)\right)-\overline H(p)+\ep v_p(0)-\ep v_p\left(\eta\left(-\ep^{-1}\right)\right)\right\}\\
&\geq \sup_{|p| \leq C_1}\left\{p\cdot \left(-\ep\eta\left(-\ep^{-1}\right)\right)-\overline H(p)+\ep v_p(0)-\ep v_p\left(\eta\left(-\ep^{-1}\right)\right)\right\}\\
&\geq \overline L\left(-\ep\eta\left(-\ep^{-1}\right)\right)-C\ep.
\end{align*}
We combine the above with  (\ref{LH-new}) to yield
\begin{align*}
u^{\ep}(0,1)&\geq \inf_{\substack{\eta(0)=0, \\ \ep\left|\eta (-\ep^{-1})\right|\leq C}}\left\{g\left(\ep\eta \left(-\ep^{-1}\right)\right)+\overline L\left(-\ep\eta\left(-\ep^{-1}\right)\right) \right\}-C\ep\\
&\geq \inf_{y\in \R^n}\left\{g(y)+\overline L(-y)\right\}-C\ep\\
&=u(0,1)-C \ep.
\end{align*}
The last equality in the above holds thanks to the Hopf-Lax formula for $u$.

\end{proof}

We now proceed to prove upper bounds \eqref{upper-bound} and \eqref{upper-bound-C2}. 
We first give a definition on backward characteristics (see calibrated curves in \cite{Fa}).  

\begin{defn} \label{def:back}
For each $p\in  \R^n$,  let $v_p \in \Lip(\T^n)$ be a viscosity solution to the cell problem {\rm (E)$_p$}.
Then, $\xi: (-\infty, 0]\to \R^n$ is called a backward characteristic of $v_p$ if 
\begin{equation}\label{char}
p\cdot \xi (t_1)+v_p(\xi (t_1))-p\cdot \xi (t_2)-v_p(\xi (t_2))=\int_{t_2}^{t_1}L(\xi (t),\dot \xi(t))+\overline H(p)\,dt
\end{equation}
for all $t_2<t_1\leq 0$. 

If $\xi$ is defined on $\R$ and the above equality holds for all $t_1, t_2\in \R$ with $t_2<t_1$, $\xi$ is called a global characteristic. 
\end{defn}
 Note that for a curve $\xi$ to be a backward characteristic (or global characteristic), it suffices to have (\ref{char}) for  $t_1=0$ and a sequence of $t_{2m}\to -\infty$ as $m\to +\infty$  (or two sequences $t_{1m}\to  +\infty$ and $t_{2m}\to -\infty$ as $m\to +\infty$). 

By approximating $H$ with smooth and strictly convex Hamiltonians, it is easy to show that for each $p\in \R^n$, 
there exists a viscosity solution $v_p \in \Lip(\T^n)$ of (E)$_p$ such that for every $y\in  \R^n$, 
there is a backward characteristic $\xi_y$ of $v_p$ with $\xi_y(0)=y$. 
The following lemma is a key step toward proving \eqref{upper-bound} and \eqref{upper-bound-C2}.

\begin{lem}\label{lem:rate-connection}
Fix $(x,t) \in \R^n \times (0,\infty)$.  Assume that $u$ is differentiable at  $(x,t)$ and $\ol{H}$ is differentiable at $p$  for $p=Du(x,t)$. 
Suppose that there exist a viscosity solution  $v_p \in \Lip(\T^n)$ of {\rm (E)$_{p}$}
and a backward characteristic $\xi:(-\infty, 0]\to \R^n$ of $v_p$ such that, for some given $C_p>0$ and  $\alpha\in (0,1]$, 
\[
\left |{\xi(s)-\xi(0)\over s}-D\ol H(p)\right|\leq {C_p \over |s|^{\alpha}}  \qquad \text{for all $s<0$}.
\]
Then
\begin{equation}\label{rate-Lip}
u^\ep(x,t) \leq u(x,t) + C C_p t^{1-\alpha} \ep^{\alpha} + C\ep.
\end{equation}
If we further assume that the initial data $g\in C^2(\R^n)$ with $M=\|D^2g\|_{C(\R^n)} <\infty$, then the above bound can be improved to 
\begin{equation}\label{rate-2-alpha}
u^\ep(x,t) \leq u(x,t) + M C_p^2 t^{2(1-\alpha)} \ep^{2\alpha} + C\ep.
\end{equation}
\end{lem}

\begin{proof}  Note that $\|Du\|_{L^{\infty}(\R^n\times [0, \infty))}\leq \|Dg\|_{L^{\infty}(\R^n)}$. 
It suffices to prove the above for $(x,t)=(0,t)$.   
By the Hopf-Lax formula, 
\begin{align*}
u(0,t)&=\min_{y\in  \R^n}\left\{g(y)+t\overline L(-t^{-1}y)\right\}\\
&=g(y_0)+t\overline L(-t^{-1}y_0)
\end{align*}
for some $y_0 \in \R^n$.
Then $p=Du(0,t)\in  \partial \overline L(-t^{-1}y_0)$. Hence, $-t^{-1}y_0=D\overline H  (p)$ and also
\[
t \overline L(-t^{-1}y_0)=- y_0\cdot p - t\overline H(p).
\]
Let $v_p$ and $\xi$ be the viscosity solution and its backward characteristic from the assumption. 
By periodicity, we may assume that $\xi(0) \in [0,1]^n$.
By our assumption, 
\[
|y_0-\ep \xi(-\ep^{-1}t)+\ep\xi(0)|\leq C_p t^{1-\alpha}\ep^{\alpha},
\]
and hence
\[
|y_0 - \ep \xi(-\ep^{-1}t)| \leq C_p t^{1-\alpha}\ep^{\alpha} + C\ep.
\]
We use the above and optimal control formula of $u^\ep(0,t)$ to compute that
\begin{align*}
u^{\ep}(0,t)&\leq u^\ep(\ep\xi(0),t) + C\ep \leq g\left(\ep \xi \left({-\ep^{-1}}t\right)\right)+\ep\int_{-\ep^{-1}t}^{0}L( \xi(s),\dot \xi (s))\,ds+C\ep\\
&=g\left(\ep \xi\left({-\ep^{-1}}t\right)\right)-t\overline H(p)+p\cdot \left(-\ep \xi\left(-\ep^{-1}t\right)\right)+ p\cdot (\ep \xi(0))\\
& \qquad \qquad \qquad \qquad  \qquad \qquad \qquad \qquad  \quad   -\ep v_p \left( \xi\left(-\ep^{-1}t\right)\right)+\ep v_p(\xi(0))+C\ep\\
&\leq g(y_0)+(-y_0)\cdot p-t\overline H(p)+C C_{p} t^{1-\al} {\ep}^{\alpha} + C\ep\\
&= u(0,t)+CC_p t^{1-\alpha} \ep^{\alpha} + C\ep.
\end{align*}

Next we prove \eqref{rate-2-alpha}. If $g\in C^2(\R^n)$, then $p=Dg(y_0)$.  
Accordingly, we can refine the above calculation as following 

\begin{align*}
u^{\ep}(0,t)&\leq g\left(\ep \xi\left({-\ep^{-1}}t\right)\right)-t\overline H(p)+p\cdot \left(-\ep \xi\left(-\ep^{-1}t\right)\right)+ p\cdot (\ep \xi(0))\\
& \qquad \qquad \qquad \qquad  \qquad \qquad \qquad \qquad  \quad   -\ep v_p \left( \xi\left(-\ep^{-1}t\right)\right)+\ep v_p(\xi(0))+C\ep\\
&\leq g\left(\ep \xi\left({-\ep^{-1}}t\right)\right) + Dg(y_0)\cdot \left(-\ep \xi\left(-\ep^{-1}t\right)\right) - t\overline H(p)+ C\ep\\
&\leq g(y_0)+Dg(y_0)\cdot (-y_0) + {M\over 2}|y_0 - \ep \xi(-\ep^{-1}t)|^2 -t\overline H(p)+ C\ep\\
&\leq g(y_0)+p\cdot (-y_0)   -t\overline H(p)+M C_{p}^2 t^{2(1-\al)} {\ep}^{2\alpha} + C\ep\\
&= u(0,t)+MC_{p}^2 t^{2(1-\al)} {\ep}^{2\alpha} + C\ep.
\end{align*}

\end{proof}

\begin{rem}\label{rem:improve}

From the above argument,  it is easy to see that to obtain \eqref{rate-2-alpha},  the $C^2$ assumption can be relaxed to semi-concave assumption on $g$, i.e., $g(x)-C|x|^2$ is concave for some $C>0$. 

Moreover, if $\overline H$ is not differentiable at $p=Du(x,t)$, then $q_0=-t^{-1} y_0 \in \partial \ol H(p)$.  For each $\ep>0$,  we may replace  the backward characteristic by an ``approximate characteristic" $\xi_{\ep}: \left[-{t\over \ep}, 0\right]\to \R^n$ satisfying that
\[
\int_{-{t\over \ep}}^{0}L(\xi_{\ep}, \dot  \xi_{\ep})+\overline H(p)\,ds\leq p\cdot \left(\xi_{\ep}(0)-\xi_{\ep}\left(-{t\over \ep}\right)\right)+C_p,
\]
and
\[
\left |{\xi_\ep(s)-\xi_\ep(0)\over s}-q_0\right|\leq {C_p \over |s|^{\alpha}}  \qquad \text{for $s=-\frac{t}{\ep}$}.
\]
Sometimes, an ``approximate characteristic" can be constructed through proper combination of several backward characteristics associated with the same $p$. See Case 2 in the proof of Theorem \ref{Hedlund-optimal} for the Hedlund example. 
\end{rem}

Since $\|Du\|_{L^\infty(\R^n \times [0,\infty))} = \|Dg\|_{L^{\infty}(\R^n)}$ and $u$ is differentiable a.e. in $\R^n \times (0,\infty)$,  
by Lemma \ref{lem:rate-connection} and approximations, we have that 
\begin{cor}\label{cor:optimal-rate}
Assume that $\ol{H} \in C^1(\R^n)$.
Assume further that for every $|p|\leq \|Dg\|_{L^{\infty}(\R^n)}$,  there exist a viscosity solution  $v_p \in \Lip(\T^n)$ of {\rm (E)$_{p}$} and a backward characteristic $\xi:(-\infty, 0]\to \R^n$ of $v_p$ such that, for some $C>0$ independent of $p$, 
\[
\left |{\xi(s)-\xi(0)\over s}-D\ol H(p)\right|\leq {C \over |s|}  \qquad \text{for all $s<0$}.
\]
Then
\begin{equation}
u^\ep(x,t) \leq u(x,t) +C\ep \quad \text{for all $(x,t)\in \R^n\times [0, \infty)$}.
\end{equation}
\end{cor}
\medskip

It is well-known that, for a given backward characteristic $\xi$ of $v_p$,
 \begin{equation}\label{conv-rot-1}
{\xi(s)-\xi(0)\over s} \to q \in  \partial \overline H(p) \quad{ as } \quad  s \to -\infty,
 \end{equation}
 by passing to a subsequence if necessary.
 See \cite{WE, EG, Fa, Go}. However, a convergence rate of \eqref{conv-rot-1} is in general not available without extra assumptions.  
 The following lemma says that if   $\overline H$  is twice differentiable at  $p$, 
 then the exponent $\alpha$ in the above lemma is at least ${1\over 2}$.
See \cite[Theorem 1]{Go} for a similar result. The proof is a simple refinement of that of \eqref{conv-rot-1}.

\begin{lem}\label{lem:rotation-rate} 
Assume $\ol{H}$ is twice differentiable at $p\in \R^n$, that is, $p\in S$.  
Let  $\xi :(-\infty, 0]\to \R^n$ be a  backward characteristic of $v_p\in \Lip(\T^n)$, a solution of {\rm (E)$_p$}. 
Then, there exists $C_p>0$ depending on $H, \ol{H}, p$ and $\|Dg\|_{L^\infty(\R^n)}$ such that
\[
\left |{\xi(t)-\xi(0)\over t}-D\overline H(p)\right|\leq {C_p\over \sqrt{|t|}}  \quad \text{for $t <0$}.
\]
\end{lem}

\begin{proof}  
It suffices to consider $t\leq -1$. Set
\[
w={{\xi (t)-\xi(0)\over t}}-D\overline H(p).
\]
There is nothing to prove if $w=0$. We therefore may assume that $w \neq 0$.

By the definition of backward characteristics,
\[
p\cdot \xi(0)-p\cdot \xi(t)+v_p(\xi (0))-v_p(\xi (t))=\int_{t}^{0}L( \xi(s),\dot \xi(s))+\overline H(p)\,ds.
\]
For $\tilde p\in  \R^n$, let $v_{\tilde p} \in \Lip(\T^n)$ be a solution to (E)$_{\tilde p}$.  It is clear that 
\[
\tilde p\cdot \xi(0)-\tilde p\cdot \xi(t)+v_{\tilde p}(\xi (0))-v_{\tilde p}(\xi (t))\leq \int_{t}^{0}L( \xi(s),\dot\xi(s) )+\overline H(\tilde p)\,ds.
\]
Accordingly, for $|\tilde p-p|\leq 1$, 
\begin{equation}\label{conv-rot-2}
\overline H(\tilde p)-\overline H(p)\geq  {(\tilde p-p)\cdot {\xi (t)-\xi(0)\over t}}-{{C(1+|p|)}\over |t|}.
\end{equation}
Since $H$ is twice differentiable at $p$, there exists a constant $C_p>0$ such that, for $|\tilde p-p|\leq 1$, 
\[
\ol{H}(\tilde p)  \leq \ol{H}(p) + D\ol{H}(p)\cdot (\tilde p - p) + C_p |\tilde p - p|^2.
\]
Combine this with the above inequality to deduce that
\[
C_p |\tilde p-p|^2\geq   (\tilde p-p)\cdot \left( {{\xi (t)-\xi(0)\over t}}-D\overline H(p)\right)- {{C(1+|p|)}\over |t|}.
\]
Choose $\tilde p=p+{1\over \sqrt{|t|}}{w\over  |w|} $ to complete the proof.
\end{proof}

We are now ready to obtain \eqref{upper-bound} and \eqref{upper-bound-C2}.

\begin{proof}[Proof of upper bounds \eqref{upper-bound} and \eqref{upper-bound-C2}]   
Inequalities \eqref{upper-bound} and \eqref{upper-bound-C2} follow immediately from Lemmas \ref{lem:rate-connection} and \ref{lem:rotation-rate}.
\end{proof}

\subsection{Improved upper bound}

We now give some further situations that the upper bound  of $u^\ep -u$ can be improved to $O(\ep)$ fully.
The first situation concerns the flat part of $\ol{H}$.

\begin{cor}\label{cor:flat}  
Fix $(x,t)\in \R^n \times (0,\infty)$.
Assume that $u$ is differentiable at $(x,t)$ with $p_0=Du(x,t)$.
Assume that $p_0$ is an interior point of $F=\left\{\ol{H}=\min \ol H\right\}$, the flat part of $\ol{H}$.

{\rm (i)} If we have further that 
\begin{equation}\label{mm-equal}
\max_{y\in \T^n} \min_{p\in \R^n} H(y,p)= \min_{p\in \R^n}  \max_{y\in \T^n} H(y,p),
\end{equation}
 then
\[
|u^{\ep}(x,t)-u(x,t)|\leq C\ep.
\]
Here and below, $C$ depends only on $H, \|Dg\|_{L^\infty(\R^n)}$.

{\rm (ii)} If we do not have \eqref{mm-equal}, then we still have
\[
|u^{\ep}(x,t)-u(x,t)|\leq \frac{C\ep}{\del},
\]
where $\del = {\rm dist}(p_0, \partial F)>0$.
\end{cor}

\begin{proof}  
We only need to show that  $u^{\ep}(x,t)\leq u(x,t)+C\ep$.  

Let us first prove (i). From the cell problem (E)$_p$ for $p\in \R^n$, it is easy to see that
\[
\max_{y\in \T^n} \min_{p\in \R^n} H(y,p)\leq \min_{p\in\R^n} \ol H(p)\leq \min_{p\in \R^n}  \max_{y\in \T^n} H(y,p).
\]
Also, since $\min_{p \in \R^n} H(y,p) = -L(y,0)$ for each $y\in \T^n$, it is clear that
\[
\max_{y\in \T^n} \min_{p\in \R^n} H(y,p)=\max_{y\in \T^n}(-L(y,0)).
\]
Since $\ol H(p_0)=\min_{p\in\R^n} \ol H(p)$ and \eqref{mm-equal} holds, we have that
\[
\ol H(p_0) =\max_{y\in \T^n} \min_{p\in \R^n} H(y,p) =-L(x_0,0)
\]  
for some  $x_0\in[0,\,1]^n$. 
Note   that $D\ol H(p_0)=0$.  
Set 
$\xi(t)=x_0$ for $t\leq 0$, and observe that $\xi$ is a backward characteristic of  any solution 
$v_{p_0}$ of (E)$_{p_0}$. 
We then use the fact that
\[
\left| \frac{\xi(t)-\xi(0)}{t} - D\ol{H}(p_0) \right|  =0,
\]
and \eqref{rate-Lip} to conclude.

We now prove (ii). 
For this, we repeat the proof of Lemma \ref{lem:rotation-rate} and note that, for $|\tilde p - p_0 | \leq \del$,
\[
0 = \ol{H}(\tilde p) - \ol{H}(p_0) \geq (\tilde p-p_0)\cdot \left( {{\xi (t)-\xi(0)\over t}}\right)- {{C(1+|p_0|)}\over |t|}.
\]
Thus,
\[
\left |{\xi(t)-\xi(0)\over t}\right| \leq  \frac{C(1+|p_0|)}{\del|t|} \leq \frac{C}{\del |t|},
\]
which, together with Lemma  \ref{lem:rate-connection}, yields the conclusion.
\end{proof}

\begin{rem}\label{rem:flat}
In Corollary \ref{cor:flat}(i), we in fact only need to assume that $p_0 \in F$ such that $\ol{H}$ is differentiable at $p_0$.
Of course, when $p_0$ is an interior point of $F$, we automatically have that $\ol{H}$ is differentiable at $p_0$ and $D\ol{H}(p_0)=0$.

Also, an example of $H$ satisfying \eqref{mm-equal} is a separable Hamiltonian of the form $H(y,p)=F(p)+V(y)$ for $(y,p) \in \R^n \times \R^n$. 

\end{rem}

If initial data $g$ is concave, then the optimal convergence rate $O(\ep)$ can be derived in a different way. 
In this situation, we only assume that $g\in \Lip(\R^n)$ and do not require that $g$ is bounded.
\begin{thm}\label{thm:concave}
Suppose that $g\in \Lip(\R^n)$ is concave. Then
\[
|u^\ep(x,t)-u(x,t)|\leq C\ep \quad \text{ for all } (x,t) \in \R^n \times [0,\infty).
\]
The constant $C$ depends on $H, \|Dg\|_{L^\infty(\R^n)}$.
\end{thm}
We give a PDE proof here, which might be of independent interest.

\begin{proof}  
Since $g$ is Lipschitz continuous and concave, we write that
\[
g(x)=\inf_{z\in \mathbb{R}^n}\left\{p_z\cdot (x-z)+g(z)\right\}.
\]
Here $p_z\in D^+ g(z)$ for $z\in \mathbb R^n$. 
It is important noting that, as $g\in \Lip(\R^n)$, we always have, for $z\in \mathbb R^n$,
\[
|p_z| \leq \|Dg\|_{L^\infty(\R^n)}.
\]
Denote by
\[
{\tilde u}^{\ep}(x,t)=\inf_{z\in \mathbb{R}^n}\left\{p_z\cdot (x-z)+g(z)+\ep v\left({x\over \ep},p_z\right)-\ol H(p_z)t\right\}.
\]
Then, ${\tilde u}^\ep$ is  a viscosity solution to (C)$_\ep$ with initial data
\[
\tilde u^\ep (x,0) = \inf_{z\in \mathbb{R}^n}\left\{p_z\cdot (x-z)+g(z)+\ep v\left({x\over \ep},p_z\right)\right\}
\]
 as it is the $\inf$ of a family of solutions to (C)$_\ep$ and $H$ is convex.
Here, $v(y,p_z)$ is a viscosity to the cell problem  (E)$_{p_z}$ with $v(0,p_z)=0$ for $z\in \mathbb R^n$.  
Since
\[
\|g-{\tilde u}^{\ep}(\cdot,0)\|_{L^\infty(\R^n)}\leq C\ep,
\]
we use the comparison principle to imply
\[
|u^{\ep}(x,t)-{\tilde u}^{\ep}(x,t)|\leq C\ep \quad \text{for all $(x,t)\in \R^n\times [0, \infty)$}.
\]
Finally, by convexity of $\ol{H}$, we note
\[
u(x,t)=\inf_{z\in \mathbb{R}^n}\left\{p_z\cdot (x-z)+g(z)-\ol H(p_z)t\right\}
\]
to conclude the proof.
\end{proof}

\subsection{Hedlund's Example}

When $n\geq 3$, the only well-understood and  interesting example is the classical Hedlund example (see  \cite[Section 5]{Ba1989} for details).  
Let us consider the simplest case in three dimensions ($n=3$)
\begin{equation}\label{Hedlund}
H(y,p)={1\over a (y)}|p| \quad \text{ for all } (y,p) \in \R^3 \times \R^3,
\end{equation}
where $a: \R^3\to  \left[\delta, {1+\delta}\right]$ is a smooth $\Z^3$-periodic function satisfying

\begin{itemize}
\item[(i)] $a\geq 1$ outside $U_{\delta}(\mathcal{L})$ and $\min_{ \R^3}a={\delta }$;

\item[(ii)] $a(y)={\delta}$ if and only if $y\in  \mathcal {L}$.
\end{itemize}
Here,
\[
\mathcal{L}=\bigcup_{i=1}^{3}\,\left(l_i+\Z^3 \right)
\]
where $l_1=\R\times \{0\}\times \{0\}$, $l_2=\{0\}\times \R\times \{{1\over 2}\}$ and $l_3=\{{1\over 2}\}\times \{{1\over 2}\}\times \R$ 
are straight lines in $\R^3$. 
The constant $\delta\in (0, 10^{-2})$ is fixed, and $U_{\delta}$ is the Euclidean $\delta$-neighborhood of $\mathcal {L}$, 
which is basically the union of tubes. 
For $1 \leq i \leq 3$, each $l_i$ is a minimizing geodesic for the Riemannian metric $ds^2=a(y)^2\sum_{i=1}^{3}dy_{i}^{2}$. In particular, they are also trajectories in Mather sets.

\begin{figure}[h]
\begin{center}

\begin{tikzpicture}[x={(-0.2cm,-0.4cm)}, y={(1cm,0cm)}, z={(0cm,1cm)}]
\draw[-latex, densely dashed] (0,0,0) coordinate(O) -- (5,0,0) node[above]{$y_1$};
\draw[-latex, densely dashed] (0,0,0) coordinate(O) -- (0,3,0) node[above]{$y_2$};
\draw[-latex, densely dashed] (0,0,0) coordinate(O) -- (0,0,3) node[right]{$y_3$};
\draw (2,0,0)--(2,2,0)--(0,2,0)--(0,2,2)--(2,2,2)--(2,2,0);
\draw (2,0,0)--(2,0,2)--(0,0,2)--(0,2,2)--(2,2,2)--(2,0,2);

\draw[red] (3,0,0) ellipse (0.1cm and 0.08cm);
\draw[red, densely dashed] (-2,0,0) ellipse (0.1cm and 0.08cm);
\draw[red, densely dashed] (1,0,0) ellipse (0.1cm and 0.08cm);
\draw[red, densely dashed] (-2,0.1,0)--(3,0.1,0);
\draw[red, densely dashed] (-2,-0.1,0)--(3,-0.1,0);
\draw[red] (2,0.1,0)--(3,0.1,0);
\draw[red] (1.7,-0.1,0)--(3,-0.1,0);

\draw[cyan, densely dashed ] (0,-1,1) ellipse (0.08cm and 0.1cm);
\draw[cyan, densely dashed ] (0,0,1) ellipse (0.08cm and 0.1cm);
\draw[cyan,  densely dashed] (0,2.5,1) ellipse (0.08cm and 0.1cm);
\draw[cyan, densely dashed] (0,-1,0.9)--(0,2.5,0.9);
\draw[cyan, densely dashed] (0,-1,1.1)--(0,2.5,1.1);
\draw[cyan] (0,2,0.9)--(0,2.5,0.9);
\draw[cyan] (0,2,1.1)--(0,2.5,1.1);
\draw[cyan] (0,-1,0.9)--(0,-0.4,0.9);
\draw[cyan] (0,-1,1.1)--(0,-0.4,1.1);

\draw[blue, densely dashed] (1,1,-1) ellipse (0.1cm and 0.08cm);
\draw[blue, densely dashed] (1,1,0) ellipse (0.1cm and 0.08cm);
\draw[blue] (1,1,3) ellipse (0.1cm and 0.08cm);
\draw[blue, densely dashed] (1,1.1,-1)--(1,1.1,3);
\draw[blue, densely dashed] (1,0.9,-1)--(1,0.9,3);
\draw[blue] (1,1.1,-1)--(1,1.1, -0.4);
\draw[blue] (1,0.9,-1)--(1,0.9, -0.4);
\draw[blue] (1,0.9,2)--(1,0.9,3);
\draw[blue] (1,1.1,2)--(1,1.1,3);

\end{tikzpicture}
\caption{Shape of $U_\delta(\cL)$} \label{fig:Hedlund}
\end{center}
\end{figure}

 It was proved in \cite{Ba1989} (in the dual form of the stable norm)  that
\[
\overline {H}(p)={1\over \delta}\max\left\{|p_1|,  |p_2|,  |p_3|\right\}.
\]
See  Lemma \ref{lem:formula-Hbar} below for a similar formula.

Note that by the explanation at the beginning of this section,  $H$ can be treated as a superlinear Hamiltonian 
(that is, we can replace $H$ by a superlinear  $\tilde H$ such that $\tilde H\geq H$ and $\tilde H=H$ when $|p|\leq C$ for a suitable constant $C$). Clearly, the corresponding Lagrangian $L$ satisfies that 
\begin{equation}\label{L-Hedlund}
L(y,q)=0  \quad \text{ for  $|q|\leq {1\over a(y)}$}.
\end{equation}
Define, for $t\in \R$,
\[
\xi_1(t)=\left(\frac{t}{\del},0,0\right),  \quad \xi_2(t)=\left(0, \frac{t}{\del}, {1\over 2}\right),  \quad \xi_3(t)=\left({1\over 2}, {1\over 2}, \frac{t}{\del}\right).
\]
Although the Hedlund example demonstrates the limitation of the Aubry-Mather theory in higher dimensions, 
we are still able to obtain the optimal convergence rate  for the corresponding homogenization problem. 

\begin{thm}\label{Hedlund-optimal} 
Let $H$ be given by \eqref{Hedlund} with $a(\cdot)$ satisfying {\rm (i)--(ii)} listed right below \eqref{Hedlund}
and $g\in \BUC(\R^3) \cap \Lip(\R^3)$. Then 
\[
\|u^{\ep}-u\|_{L^\infty(\R^3 \times [0,\infty))} \leq C\ep.
\]
Here, $C$ is a  constant depending only on $a(\cdot)$ and $\|Dg\|_{L^{\infty}(\R^3)}$. 
\end{thm}

\begin{proof}  
Fix $(x,t)\in \R^3 \times (0,\infty)$.
 It suffices to show that 
 \begin{equation}\label{Hedlund-goal}
 u^{\ep}(x,t)-u(x,t)\leq C\ep.
 \end{equation}
 By approximation, we may assume that $g\in C^1$,  $u$ is differentiable at $(x,t)$. 
 Denote $p=Du(x,t)$. There are two cases to be considered.

\smallskip

\noindent {\bf Case 1.}  $\overline H$ is differentiable at $p$. Without loss of generality, we assume that 
\[
D\overline H(p)=\left({1\over \delta},0,0\right).
\]
Then, \eqref{Hedlund-goal} follows immediately from 
\[
\int_{-s}^{0}L(\xi_1, \dot \xi_1)+\overline H(p)\,ds=\int_{-s}^0\left (0+{|p_1|\over \delta}\right)\,ds=p\cdot \left(\xi_1  (0)-\xi_1\left(-s)\right)\right) \quad \text{for all $s\geq 0$},
\]
and Lemma \ref{lem:rate-connection}. Note that $\xi_1$ is a global characteristic of any corrector $v_p$ associated with $p$ (see Definition \ref{def:back} for definition of global characteristics).

\smallskip

\noindent {\bf Case 2.}  $\overline H$ is not differentiable at $p$. 
We may assume that  $p=(1,1,0)$. 
The proof for other $p$ is the same. Then
\[
\partial \overline H(p)=\left\{\frac{1}{\del}(\tau, 1-\tau, 0)\,:\,  \tau \in  [0,1]\right\}.
\]
According to Remark \ref{rem:improve}, it suffices to construct an  ``approximate backward characteristic"  $\xi_\tau: [-{t\over \ep},0]\to \R^3$ for a given $q_\tau=\frac{1}{\del}(\tau, 1-\tau, 0)\in \partial \overline H(p)$ such that 

\begin{equation}\label{Hedlund-r1}
\left|{\xi_\tau(s)-\xi_{\tau}(0)\over s}-q_\tau\right|\leq {C\over |s|} \quad \text{for $s=-\frac{t}{\ep}$},
\end{equation}
and
\begin{equation}\label{Hedlund-r2}
\int_{-{t\over \ep}}^{0}L(\xi_\tau, \dot \xi_\tau)+\overline H(p)\,ds\leq p\cdot \left(\xi_\tau  (0)-\xi_\tau\left(-{t\over \ep}\right)\right)+C.
\end{equation}
The idea is quite simple. 
We build $\xi_\tau$ by combining a part of $\xi_1$ with a part of a proper translation of $\xi_2$.
Let $k \in \Z$ be the integer part of $-\frac{\tau t}{\ep \del}$.
\[
\xi_\tau(s)=
\begin{cases}
\left(\frac{s}{\del}, 0, 0\right)  \quad &\text{for $s\in  [- {\tau t\over \ep}, 0]$,}\\
\left(s+\frac{\tau t}{\ep}\right) \left( \left(-\frac{\tau t}{\ep \del}, 0,0\right) - \left(k,0,\frac{1}{2}\right) \right) +  \left(-\frac{\tau t}{\ep \del}, 0,0\right) \quad &\text{for $s\in  [- {\tau t\over \ep}-1, - {\tau t\over \ep}]$,}\\
\left(k,\frac{1}{\del}\left(s+\frac{\tau t}{\ep} +1 \right), {1\over 2}\right) \quad &\text{for $s\in  [-\frac{t}{\ep}, - {\tau t\over \ep}-1]$.}
\end{cases}
\]
It is clear from the construction that \eqref{Hedlund-r1} holds.
Let us now verify \eqref{Hedlund-r2}. Due to \eqref{L-Hedlund} and $a=\delta$ on ${\mathcal{L}}$, we compute 
\[
\int_{-{\tau t\over \ep}}^{0}L(\xi_\tau, \dot \xi_\tau)+\overline H(p)\,ds=\frac{\tau t}{ \ep \del}= p \cdot \left(\xi_\tau  (0))-\xi_\tau\left(-{\tau t\over \ep}\right)\right),
\]
and
\[
\int_{-{t\over \ep}}^{-{\tau t\over \ep}-1}L(\xi_\tau, \dot \xi_\tau)+\overline H(p)\,ds = \frac{1}{\del} \left( \frac{(1-\tau)t}{\ep} - 1 \right)
=p \cdot \left(\xi_\tau\left(-{\tau t\over \ep}-1\right)-\xi_\tau\left(-{t\over \ep}\right)\right).
\]
Then \eqref{Hedlund-r2} follows from the fact that $\left|\xi_\tau\left(-{\tau t\over \ep}-1\right)-\xi_\tau\left(-{\tau t\over \ep}\right)\right|\leq 2$ and $\left|\dot \xi_{\tau}(s)\right|\leq 2$ for $s\in  [- {\tau t\over \ep}-1, - {\tau t\over \ep}]$.

\begin{figure}[h]
\begin{center}

\begin{tikzpicture}[x={(-0.2cm,-0.4cm)}, y={(1cm,0cm)}, z={(0cm,1cm)}]

\draw[red] (3,0,0) ellipse (0.1cm and 0.08cm);
\draw[red, densely dashed] (-2,0,0) ellipse (0.1cm and 0.08cm);
\draw[red, densely dashed] (1,0,0) ellipse (0.1cm and 0.08cm);
\draw[red] (-2,0.1,0)--(3,0.1,0);
\draw[red] (-2,-0.1,0)--(3,-0.1,0);

\draw[cyan, densely dashed ] (0,-1,1) ellipse (0.08cm and 0.1cm);
\draw[cyan, densely dashed ] (0,0,1) ellipse (0.08cm and 0.1cm);
\draw[cyan,  densely dashed] (0,2.5,1) ellipse (0.08cm and 0.1cm);
\draw[cyan] (0,-1,0.9)--(0,2.5,0.9);
\draw[cyan] (0,-1,1.1)--(0,2.5,1.1);

\draw[thick] (5,0,0)--(0.5,0,0)--(0,0,1)--(0,-2.5,1) node[above]{$\xi_\tau$};
\draw[->] (0,-2.5,1)--(0,-1.5,1);
\draw[->] (0,0,1)--(0.25,0,0.5);
\draw[->] (0.5,0,0)--(4,0,0);
\draw (5,0,0) node[below]{$0$};

\end{tikzpicture}
\caption{Approximate backward characteristic $\xi_\tau$} \label{fig:xi-tau}
\end{center}
\end{figure}
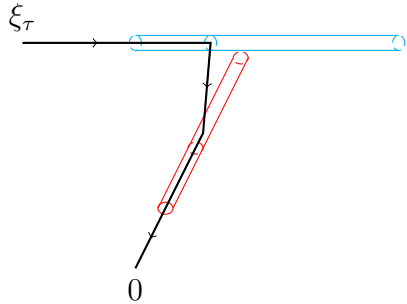

\end{proof}

\begin{rem}\label{rem:ndgeneral}  
In general,  unlike the two dimensional case (see Appendix), very little is known about the structure of Aubry and Mather sets when $n\geq 3$  due to the lack of topological restrictions. 
Identifying the shape of $\overline H$ is extremely challenging. 
Again, for $n\geq 3$, consider the basic form of a  coercive, convex and positively homogeneous Hamiltonian of degree $1$ as following
\[
H(y,p)=\frac{1}{a(y)} |p| \quad \text{ for all } (y,p) \in \R^n \times \R^n.
\]
Here,  $a(\cdot)$ is  a positive, smooth and $\Z^n$-periodic function.  
It is clear that  the corresponding effective Hamiltonian $\overline H(p)$ is also coercive, convex and positively  homogeneous of degree $1$. 
However, it is very hard  to derive any further information about $\overline H$.  
For instance, it is not even known whether $\overline H$ could be isotropic, that is, $\ol{H}(p)=|p|$ for all $p\in \R^n$, 
if $a(\cdot)$ is not constant (the answer is  negative in two dimensions \cite{Ba2}). 

By the Hopf-Lax formula, for $(x,t)\in \R^n \times (0,\infty)$,
\[
u(x,t)=\min_{y\in D_0}g(x-ty),
\]
where $D_0=\partial \overline H(0)$. 
It seems necessary to add certain technical assumptions about geometric properties of $D_0$ or the underlying Hamiltonian systems in order to improve  (ii) in Theorem \ref{thm:nd} when $n\geq 3$ 
(e.g., obtaining $O(\ep)$ for a.e. $(x,t) \in \R^n \times (0,\infty)$). 
The key question is what kind of assumptions could be considered as ``reasonable" and ``generic" even if they are in general hard to verify. 
We plan to investigate  this in the future. 

\end{rem}


\section{Proof of Theorem \ref{thm:2d}} \label{sec:2d}
In this section, the assumptions of Theorem \ref{thm:2d} are always in force.
By approximation, we may assume that, for each $y \in  \R^2$, $p \mapsto H(y,p)$ is smooth away from the origin, and for each $y\in \T^2$, the level curve
\[
\{p\in \R^2\,:\,  H(y,p)=1\}
\]
is strictly convex, i.e., the curvature is positive.  
It is worth mentioning that constant $C$ in Theorem \ref{thm:2d} does not depend on the smoothness of $H$. 
Denote
\[
S_1=\{p \in \R^2\,:\, \overline H(p)=1\}.
\]
It is clear that $\ol{H}$ is positively homogeneous of degree $k$.
By \cite{Car},  $S_1$ is $C^1$ and therefore $\overline H$ is $C^1$ away from the origin  (see  Theorem \ref{thm:2D-level} in Appendix for a sketch of the  proof of this interesting fact). For $p\in S_1$,  let $v$ be a viscosity solution to (E)$_p$, that is, $v$ solves
\[
H(y,p+Dv)=\overline H(p)=1 \quad \text{ in } \T^2.
\]
Choose a smooth (Tonelli) Hamiltonian $\tilde H:\T^2\times \R^2\to \R$ such that $\tilde H$ satisfies (H1)--(H3), $\tilde H(y,\cdot):\R^2\to \R$ is strictly convex for all $y\in \T^2$, and
\[
\tilde H=e^{H-1}  \quad \text{when $H\geq {1\over 2}$}.
\]
Clearly,  $\ol {\tilde H}(p)=e^{\ol H(p)-1}=1$ and $v$ is also a viscosity solution to 
\[
\tilde H(y,p+Dv)=1 \quad \text{ in } \T^2.
\]
Choose $\xi:\R\to \R^2$ to be an orbit from the projected Mather set $\mathcal {M}_p$ associated with $\tilde H$. 
Here we lift $\xi$ from $\T^2$ to $\R^2$. Then $v$ is $C^1$ along $\xi$ and  for $t\in \R$, 
\[
\tilde H(\xi(t), p+Dv(\xi (t)))=H(\xi(t), p+Dv(\xi (t)))=1
\]
and
\begin{equation}\label{global-chara}
\dot \xi(t)=D_p\tilde H(\xi(t), p+Dv(\xi (t)))=D_p H(\xi(t), p+Dv(\xi (t))).
\end{equation}
Without loss of generality, we may assume that  $\xi(0)\in  [0,1]^2$.
\begin{lem}\label{lem:2d-rotation-rate}There exists a constant $C>0$ independent of $p\in S_1$ and $t \in \R$ such that
\begin{equation}\label{Hulend}
|\xi(t)-\xi (0)-tD \overline H(p)|\leq C  \quad \text{ for all $t\in \R$}.
\end{equation}
\end{lem}

\begin{proof}  It suffices to show that 
\[
|\xi(t)-tD \overline H(p)|\leq C  \quad \text{ for all $|t|\geq 1$}.
\]
Write $D\overline H(p)=(a,b)$ and $\xi(t)=(x(t),y(t))$. 
Since $p\cdot D\overline H(p)=k$,  $|D\ol H(p)|\geq {k\over |p|}$. 
Without loss of generality, we may assume that 
\[
a>0  \quad \mathrm{and} \quad {b\over a}\in [0,1].
\]
Due to the identification of $\xi$ with a circle homeomorphism (see Lemma \ref{slope-rate} in Appendix) from the Aubry-Mather theory,  we have that
\[
\left|y(t)-{b\over a}x(t)\right|\leq C.
\]
We would like to remind readers that the above inequality is derived mainly based on topological arguments. In particular, the constant $C$ does not rely on the smoothness of $H$ (see Remark \ref{rem:constant-issue}).

 Accordingly,  for $|t|\geq 1$, 
\begin{equation}\label{rotation-number}
\left|{\xi (t)\over t}-{x(t)\over t}\left(1, {b\over a}\right)\right|\leq {C\over {|t|}}.
\end{equation}
Owing to  \eqref{global-chara},
\begin{align*}
&p\cdot \xi (t)-p\cdot \xi (0)+v(\xi(t))-v(\xi (0)) = \int_0^t (p+Dv(\xi(s)))\cdot \dot \xi(s)\,ds\\
=\, & \int_{0}^{t}(p+Dv(\xi(s)))\cdot D_pH(\xi(s),p+Dv(\xi(s)))\,ds\\
=\, & kt=t p\cdot D\overline H(p).
\end{align*}
Therefore, for $|t|\geq1$,
\begin{equation}\label{non-tangential-control}
\left|p\cdot \left({\xi (t)\over t}-D\overline H(p)\right)\right|\leq {C\over {|t|}}.
\end{equation}
Combining this with \eqref{rotation-number} and the fact that 
\[
p\cdot \left(1, {b\over a}\right)={k\over a},
\]   
we have that 
\[
\left|{x(t)\over t}-a\right|\leq {Ca\over k|t|}.
\]
Again, thanks to \eqref{rotation-number},
\begin{align*}
\left|{\xi (t)\over t}-D\overline H(p)\right|&\leq \left|{x(t)\over t}\left(1, {b\over a}\right)-D\overline H(p)\right|+{C\over |t|}\\
&=\left|{x(t)\over t}-a\right| {\sqrt {a^2+b^2}\over a}+{C\over |t|}.
\end{align*}
Therefore, 
\[
\left|{\xi (t)\over t}-D\overline H(p)\right|\leq {C\over {|t|}} \left({|D\ol H(p)|\over k}+1\right)\leq {C\over {|t|}} \left({\max_{p'\in S_1}|D\ol H(p')|\over k}+1\right).
\]
Hence  \eqref{Hulend} holds.
\end{proof}

\begin{proof}[Proof of Theorem \ref{thm:2d}]
By approximation, it suffices to prove Theorem \ref{thm:2d} for $k>1$. 
Of course, we only need to prove the upper bound, that is,
\[
u^\ep(x,t) \leq u(x,t)+C\ep \quad \text{for all $(x,t)\in  \R^n\times [0, +\infty)$}.
\]
Thanks to Corollary \ref{cor:optimal-rate}, this follows immediately from the following claim. 

\medskip

{\bf Claim:} For any $p_0\in  \R^n$, let $v_0$ be a viscosity solution to (E)$_{p_0}$. Then there exists a global characteristic $\xi_0:\R\to \R^2$ associated with $v_0$ such that 
\[
\left| \xi_0(t)-\xi_0(0)-tD \overline H(p_0)\right| \leq C  \quad \text{for all $t\in \R$}.
\]
The constant $C$ is the same as that in Lemma \ref{lem:2d-rotation-rate}.
In fact,   there are two cases to be considered.

\medskip

\noindent {\bf Case 1. $p_0=0$}. Then $D\overline H(p_0)=0$. Since $L(y,0)=0$ for all $y\in \R^2$, we can simply choose $\xi_0(s)= 0$ for $s\in  \R$.

\medskip

\noindent {\bf Case 2. $p_0\not= 0$.}  Then $\overline H(p_0)>0$. Choose  $\tau>0$ satisfying that $\tau^k={\ol H(p_0)}^{-1}$. Then
\[
\overline H(\tau p_0)=1.
\] 
Let $\xi$ be a global characteristic associated with $v=\tau v_0$ and $p=\tau p_0$ as discussed in \eqref{global-chara}.  
Since $D \overline H(\tau p)=\tau^{k-1}D\ol H(p)$, we get that $\xi_0(t)=\xi (\tau^{1-k} t)$ (equivalently, $\xi(t)=\xi_0 (\tau^{k-1} t)$)  is a global characteristic associated with  $v_0$.  In light of  Lemma \ref{lem:2d-rotation-rate}, we have that 
\[
\left| \xi_0(\tau^{k-1}t)-\xi_0 (0)-tD \overline H(\tau p_0) \right| \leq C  \quad \text{for all $t\in \R$}.
\]
By using the fact that $D \overline H(\tau p_0)=\tau^{k-1}D\ol H(p_0)$, we imply
\[
\left| \xi_0(t)-\xi_0 (0)-tD \overline H(p_0)\right| \leq C  \quad \text{for all $t\in \R$}.
\]
\end{proof}

\begin{rem}\label{rem:mechanical-Ham}
The assumption of homogeneity plays a crucial role in deriving the control along a non-tangential direction in \eqref{non-tangential-control}. 
For other important Hamiltonians, e.g.,  the mechanical one $H(y,p)={1\over 2}|p|^2+V(y)$, 
a reasonable question is whether, in two dimensions,  
upper bound \eqref{upper-bound-C2} can be improved to
\[
u^{\ep}(x,t)\leq u(x,t)+C_p\sqrt{t}\ep+C\ep \quad \text{for a.e. $(x,t)\in  \R^2\times (0, +\infty)$}.
\]
Using  special structures in two dimensions, this boils down to find  a second derivative bound of $\ol H$ along at least one non-tangential direction. More precisely, given $p\in \R^2$, if $\ol H(p)>\min \ol H$ and $D\ol H(p)$ is an irrational vector, 
does there exist a unit vector $r$ such that $r \cdot D\overline H(p)\not=0$ and the function 
\[
w(s)=\ol{H}(p+sr)
\]
satisfies that $w(s)\leq w(0)+w'(0)s+Cs^2$? We will investigate this in the future.

\end{rem}


\section{Proof of Theorem \ref{thm:1d}} \label{sec:1d}

We are now in the one dimensional setting.

\begin{proof}[Proof of Theorem \ref{thm:1d}] Since the constant $C$ does not depend on the smoothness of $H$, 
by approximation, we may  assume that $H(y,p)$ is smooth and strictly convex in $p$ variable (i.e., Tonelli Hamiltonian). 
In one dimension, $\ol H$ possesses an explicit formula. In particular, we have that $\overline H\in C^1(\R)$ (see \cite{Be} for instance).

Thanks to \eqref{lower-bound} and Corollary \ref{cor:optimal-rate},  Theorem \ref{thm:1d}  follows immediately from the following lemma. 

\begin{lem} 
For $p\in  \R$, let  $v$ be a viscosity solution to
\[
H(y,p+v')= \ol{H}(p)   \quad \text{in $\T$}.
\]
Then any backward characteristic $\xi:(-\infty, 0]\to \R$ of $v$ satisfies that
\[
\left |{\xi(t)-\xi(0)\over t}-\overline H'(p)\right|\leq {1\over |t|}  \quad \text{for all $t<0$}.
\]
\end{lem}

\begin{proof} Fix $p\in \R$. There are two cases to be considered.

\medskip

\noindent {\bf Case 1.} $\overline H(p)=\min \ol H$.  Then $\ol{H}'(p)=0$. 

Let $\xi$ be a backward characteristic of $v$ with $\xi(0)=0$.
Since $\xi$ cannot intersect itself,  we have either $\xi((-\infty,0])\subset [0,\infty)$ or $\xi((-\infty,0]) \subset (-\infty, 0]$. 
Without loss of generality, we assume that  $\xi((-\infty,0])\subset [0,\infty)$. 
Note that, $\xi$ satisfies
\[
\dot \xi=H_{p}(\xi, p+v'(\xi)) \quad \text{for all $t\leq 0$}.
\]
We claim that 
\begin{equation}\label{case1-bdd}
\xi((-\infty,0])\subset [0,1].
\end{equation}
Assume otherwise that \eqref{case1-bdd} does not hold.
Then $\xi(T)=1$ for some $T<0$, and we deduce also that $v\in C^1(\T)$.
By periodicity,  $\xi (mT)=m$ for all $m\in \N$. Therefore,
\[
\lim_{t\to -\infty}{\xi(t)\over t}={1\over T}\not= 0=\ol H'(p),
\]
which is a contradiction.
Thus, \eqref{case1-bdd} holds.

\medskip

\noindent
{\bf Case 2.} $\overline H(p)>\min \ol H$.  
Without loss of generality, we assume $\ol{H}'(p)>0$. 
Let $\xi$ be a backward characteristic of $v$ with $\xi(0)=0$. 
Then $\xi((-\infty,0])\subset (-\infty,0]$, $v\in C^{1,1}(\R)$ and 
\[
\dot \xi=H_{p}(\xi, p+v'(\xi))>0 \quad \text{for all $t\leq 0$}.
\]
Then, by changing of variables $x=\xi(s)$, we imply
\[
|t|=\int_t^0 \,ds = \int_t^0 \frac{\dot \xi(s)}{\dot \xi(s)}\,ds=\int_{\xi(t)}^{0}{1\over F_1(x)}\,dx,
\]
where $F_1(x)=H_p(x, p+v'(x))$ for $x\in \R$.  Accordingly, for $t<0$,
\begin{align*}
{t\over \xi (t)}&={1 \over |\xi(t)|} \int_{\xi(t)}^{0}{1\over F_1(x)}\,dx
=\left(1+{{E_t}\over \xi(t)}\right)\int_{0}^{1}{1\over F_1(x)}\,dx,
\end{align*}
where $E_t$ is an error term satisfying $|E_t|\leq 1$ thanks to Lemma \ref{lem:class-ineq} below. 
Then
\[
\left|{\xi (t)\over t}-\left(\int_{0}^{1}{1\over F_1(x)}\,dx\right)^{-1}\right|={|E_t|\over |t|}\leq {1\over |t|}.
\]
The proof is complete and we get in addition that
\[
\ol{H}'(p)= \left(\int_{0}^{1}{1\over F_1(x)}\,dx\right)^{-1}.
\]
\end{proof}

\end{proof}

\begin{lem}\label{lem:class-ineq}
Assume that $f \in C(\T, [0,\infty))$ and $L>0$ are given. Then
\[
\left| \int_0^{L} f\,dy - L \int_0^1 f\,dy \right| \leq \int_0^1 f\,dy.
\]
\end{lem}
We can view this lemma as a quantitative version of the ergodic theorem for periodic functions in one dimension.
\begin{proof}
For a given real number $s\in \R$, denote by $[s]$ its integer part. We have
\begin{align*}
&\left| \int_0^{L} f\,dy - L \int_0^1 f\,dy \right| =\left|  \int_0^{[L]} f\,dy + \int_{[L]}^{L} f\,dy  - L \int_0^1 f\,dy \right| \\
=\,& \left| ([L] - L) \int_0^1 f\,dy +\int_{[L]}^{L} f\,dy \right| \\
\leq\,& \max\left \{  (L-  [L] ) \int_0^1 f\,dy, \int_{[L]}^{L} f\,dy  \right\} \leq \int_0^1 f\,dy.
\end{align*}
\end{proof}

 \subsection{An example on optimal rate $O(\ep)$ in one dimension} \label{subsec:ex}
We now show that $O(\ep)$ is indeed the optimal rate of convergence via the following simple proposition.
\begin{prop}\label{prop:optimal}
Assume that $n=1$ and $H(y,p) = \frac{p^2}{2} + V(y)$ for some given $V \in C(\T)$ with $\max_{\T} V=0$ and $V \leq -1$ in $[-3^{-1},3^{-1}]$.
Assume further that $g \equiv 0$. 
For $\ep>0$, let $u^\ep$ be the solution to {\rm (C)$_\ep$}.
Let $u$ be the solution to {\rm (C)}.
Then, $u^\ep$ converges locally uniformly to $u \equiv 0$ on $\R \times [0,\infty)$ as $\ep \to 0$.
Furthermore, for $\ep\in (0,1)$,
\begin{equation}\label{op-ep}
u^\ep(0,1) \geq \frac{\ep}{6}.
\end{equation}
\end{prop}

\begin{proof}
It is obvious that  $u^\ep$ converges locally uniformly to $u \equiv 0$ on $\R \times [0,\infty)$ as $\ep \to 0$.
We just need to prove \eqref{op-ep}.
Thanks to the optimal control formula,
\[
u^\ep(0,1)= \inf \left\{ \ep \int_0^{\ep^{-1}} \frac{|\dot \eta|^2}{2} - V\left(\eta\right)\,dt \,:\, \eta \in \AC([0,\ep^{-1}]), \eta(0)=0 \right\}.
\]
Pick $\eta \in \AC([0,\ep^{-1}])$ with $\eta(0)=0$. There are two cases to be considered.

Firstly, if $\eta([0,3^{-1}]) \subset [-3^{-1},3^{-1}]$, then
\[
 \ep \int_0^{\ep^{-1}} \frac{|\dot \eta|^2}{2} - V\left(\eta\right)\,dt  \geq \ep \int_0^{3^{-1}} -V(\eta)\,dt \geq \frac{\ep}{3}.
\]
If not, then without loss of generality, we may assume that there exists $t \in (0,3^{-1})$ such that $\eta(t)=3^{-1}$. Then
\[
 \ep \int_0^{\ep^{-1}} \frac{|\dot \eta|^2}{2} - V\left(\eta\right)\,dt  \geq \ep \int_0^t \frac{|\dot \eta|^2}{2}\,dt
 \geq \frac{\ep}{2t} \left(\int_0^t \dot \eta\,dt\right)^2 \geq \frac{\ep}{6}.
\]
The proof is complete.
\end{proof}


\section{Lack of continuous selection of $v(y,p)$ with respect to $p$:  why the PDE method in \cite{CDI} cannot be extended} \label{sec:non-cont}

In \cite{CDI}, the authors introduce a rigorous way to approximate the two-scale asymptotic  expansion \eqref{expansion}
by using the ergodic problems and  their discounted approximations.
More precisely, for each $p\in \R^n$, instead of using $v(y,p)$ directly, they use $v^\lam(y,p)$, solution to the following discount problem
\[
\lam v^\lam+H(y,p+Dv^\lam)=0 \quad \text{ in } \T^n,
\]
and approximate $Du(x,t)$ by ${x-y\over \ep ^{\beta}}$ in \eqref{expansion} by using the well-known doubling of variables method.
By optimizing choices of  $\lam$ and $\beta$ in their arguments,   $O(\ep^{1/3})$ is the best convergence rate obtained.

In order to improve the convergence rate by modifying this approach,  
it is necessary to have a nice selection of viscosity solutions to the cell problem (E)$_{p}$ with respect to $p$.  
This allows us to use directly $v(y,p)$ instead of $v^\lam(y,p)$ in \eqref{expansion}.
For example,  if we are able to choose for each $p\in \R^n$ a solution $v(y,p)$ of (E)$_p$ such that
\[
 p \mapsto v(y,p) \quad \text{is  Lipschitz},
\]
then the convergence rate can be improved from $O(\ep^{1/3})$ to $O(\ep^{1/2})$. 

\medskip

In this section, we present an example where there is no continuous selection of $v(y,p)$ with respect to $p$. 
Assume that $n=3$. Define
\[
\begin{cases}
L_1 = \left\{ (t,1/2,0)\,:\, 0 \leq t \leq 1 \right\},\\
L_2 = \left\{ (0, t,1/2)\,:\, 0 \leq t \leq 1 \right\},\\
L_3 = \left\{ (1/2,0,t)\,:\, 0 \leq t \leq 1 \right\}.
\end{cases}
\]

We consider the following mechanical Hamiltonian:
\begin{equation}\label{counter}
\begin{cases}
H(y,p)= \frac{1}{2}|p|^2 + V(y) \quad \text{ where } V \in C^{\infty}(\T^3),\\
\max_{\T^3} V =0 \quad \text{and} \quad \{V=0\} = L_1 \cup L_2 \cup L_3.
\end{cases}
\end{equation}
Denote by $e_1=(1,0,0), e_2=(0,1,0)$ and $e_3=(0,0,1)$.
For each $p\in \R^3$, let $\mathcal {A}_p$ be the projected Aubry set associated with (E)$_p$.
We have first the following result.

\begin{thm}\label{thm:A-e}
Assume that \eqref{counter} holds. 
There exists $\del>0$ sufficiently small such that
\[
\mathcal{A}_{\ep e_i} = L_i \quad \text{ for all } \ep \in (0,\del), \, i \in \{1,2,3\}.
\]
\end{thm}

The above theorem follows from the following stronger lemma and the constructions of critical subsolutions in its proof.

\begin{lem} \label{lem:formula-Hbar}
There exists $\del>0$ sufficiently small such that, for $|p|<\del$,
\[
\ol{H}(p) = \max\left\{ \frac{|p_1|^2}{2}, \frac{|p_2|^2}{2}, \frac{|p_3|^2}{2} \right\},
\]
where $p=(p_1,p_2,p_3)$.
\end{lem}

\begin{proof}
We have the following inf-max representation formula for $\ol{H}(p)$
\begin{equation}\label{inf-max}
\ol{H}(p) = \inf_{\phi \in C^1(\T^3)} \max_{y \in \T^3} \left( \frac{1}{2} |p+D\phi(y)|^2 +V(y) \right).
\end{equation}
For any $\phi \in C^1(\T^3)$, we have $\phi(0,1/2,0)=\phi(1,1/2,0)$.
By the Rolle theorem, there exists $s \in (0,1)$ such that $\phi_{y_1}(s,1/2,0) =0$.
Plug this into \eqref{inf-max} and use the fact that $V(s,1/2,0)=0$ to get that
\[
\ol{H}(p) \geq \frac{|p_1|^2}{2}.
\]
By repeating this simple observation two more times, we get, for all $p\in \R^3$,
\begin{equation}\label{one-side}
\ol{H}(p) \geq  \max\left\{ \frac{|p_1|^2}{2}, \frac{|p_2|^2}{2}, \frac{|p_3|^2}{2} \right\}.
\end{equation}

We now need to obtain the reverse inequality for $|p|$ small.
For $1\leq i \leq 3$, set
\[
L_{i,1/8}= \left\{ x \in \R^3\,:\, \text{dist}(x,L_i) \leq \frac{1}{8} \right\}.
\]
For $p \in \R^3$ such that $|p|=1$, we construct a smooth periodic function $\eta$ satisfying
\[
\begin{cases}
\eta(y) = -p_2 y_2 - p_3 y_3 \qquad &\text{ for } y \in L_{1,1/8},\\
\eta(y) = -p_1 y_1 - p_3 y_3 \qquad &\text{ for } y \in L_{2,1/8},\\
\eta(y) = -p_1 y_1 - p_2 y_2 \qquad &\text{ for } y \in L_{3,1/8}.
\end{cases}
\]
Set $C_1=\max_{\R^3}|D\eta|$. Choose $\del$ such that
\[
\del = \frac{1}{\sqrt{1+C_{1}^{2}}} \left (- \max_{\T^3 \setminus \cup_{i=1}^3 L_{i,1/8}} V \right)^{1/2}.
\]
Then for $\ep \in (0,\del)$, we have that $\phi^\ep = \ep \eta$ is a subsolution to
\begin{equation}\label{the-other}
\frac{1}{2} |\ep p+D\phi^\ep|^2 + V =  \max\left\{ \frac{| \ep p_1|^2}{2}, \frac{| \ep p_2|^2}{2}, \frac{|\ep p_3|^2}{2} \right\} \qquad \text{ in } \T^3.
\end{equation}
We combine \eqref{one-side} and \eqref{the-other} to yield that 
\[
\ol{H}(p) = \max\left\{ \frac{|p_1|^2}{2}, \frac{|p_2|^2}{2}, \frac{|p_3|^2}{2} \right\} \qquad \text{ for all } |p| < \del.
\]

Furthermore, for  $i \in \{1,2,3\}$ fixed and $p=\ep e_i$, $\phi_\ep$ is a strict subsolution of \eqref{the-other} on $\T^3 \setminus L_i$, i.e., 
\[
\begin{cases}
\frac{1}{2} |\ep p+D\phi^\ep|^2 +V<\ol H(p)   \quad \text{in $\T^3\backslash L_i$}\\
\frac{1}{2} |\ep p+D\phi^\ep|^2 +V=\ol H(p)   \quad \text{in $L_i$}.
\end{cases}
\]
Also, $p+D\phi^\ep=p_i e_i$ along $L_i$. We hence deduce easily that $\mathcal{A}_{\ep e_i} = L_i$ for all $\ep \in (0,\del)$.
\end{proof}

Note that for this example, the Mather set and Aubry set coincide. It is well-known (see \cite{Fa}) that solutions to the cell problem  {\rm (E)$_p$}  are determined by their values on $\mathcal {A}_p$.  
Moreover, all subsolutions are  differentiable on $\mathcal {A}_p$ with the same gradients. 
 Accordingly, for $p=\ep e_i$, if $v(x,p)$ is a viscosity solution to  {\rm (E)$_{p}$}, then
\[
Dv=D\phi^\ep  \quad \text{on $L_i$},
\]
which yields that $v$ is constant along $L_i$. An immediate corollary is
\begin{cor}\label{cor:unique}
For $\ep \in (0,\del)$ and $i \in \{1,2,3\}$,
the cell problem {\rm (E)$_{\ep e_i}$} has a unique solution (up to additive constants) which is given by
\begin{equation}\label{HL-Mather}
v(y,\ep e_i)=\inf_{\substack{t>0 \\  \xi\in \Gamma_{i,t,y}}}\left(\int_{0}^{t}{1\over 2}|\dot \xi|^2-\dot \xi\cdot \ep e_i+V(\xi)+\ol H(\ep e_i)\,ds\right)+c,
\end{equation}
for some constant $c \in \R$ with $v=c$ along $L_i$. 
Here  $\Gamma_{i,t,y}$ is the collection of all absolutely continuous curves from $[0,t]$ to $\R^3$ such that $\xi(0)\in L_i+\Z^3$ and $\xi(t)=y$.
\end{cor}
See also \cite{Fa} for formula \eqref{HL-Mather}.
Here is the main result of this section.

\begin{thm}\label{thm:non-cont}
Assume that \eqref{counter} holds.
Then there is no way to select for each $p\in \R^3$, $v(y,p)\in \Lip(\T^3)$, a solution to {\rm(E)$_p$},  such that
\begin{equation}\label{cont-select}
p \mapsto v(y,p) \text{ is continuous at } p=0.
\end{equation}

\end{thm}

\begin{proof}
Assume by contradiction that there is a way to select $v(y,p) \in \Lip(\T^3)$ such that \eqref{cont-select} holds.
Then, 
\begin{equation}\label{non-cont1}
\lim_{\ep \to 0} v(y,\ep e_1) = \lim_{\ep \to 0} v(y,\ep e_2) = \lim_{\ep \to 0} v(y, \ep e_3).
\end{equation}
Assume that for $i \in \{1,2,3\}$
\[
v(y, \ep e_i)=c_{i\ep}   \quad \text{along $L_i$}.
\]
for some constant $c_{i\ep}$. We claim that for $i \in \{1,2,3\}$,
\begin{equation}\label{non-cont2}
\lim_{\ep \to 0} v(y,\ep e_i) = d(y,L_i) + c_i,
\end{equation}
where $c_i=\lim_{\ep\to 0}c_{i\ep}$, and 
\[
d(y,L_i) =\inf_{\substack{t>0\\  \xi\in \Gamma_{i,t,y}}}\left(\int_{0}^{t}{1\over 2}|\dot \xi|^2-V(\xi)\,ds\right).
\]
It suffices to prove this for $i=1$ as proofs for $i=2,3$ are similar. It is easy to see that 
\[
\limsup_{\ep\to 0}v(y, \ep e_1)\leq d(y,L_1) + c_1.
\]
The other direction is more subtle since $t$ in formula \eqref{HL-Mather} (and also $|\xi(0)|$) could go to $+\infty$.  
We need to make use of the special structure of $V$ and the formula of $\overline H$ here.  Note that
\[
v(y,\ep  e_1)=\inf_{\substack{t>0 \\  \xi\in \Gamma_{i,t,y}}}\left(\int_{0}^{t}{1\over 2}|\dot \xi|^2-\dot \xi\cdot \ep e_1-V(\xi)+\ol H(\ep e_1)\,ds\right)+c_{1\ep}\geq c_{1\ep}.
\]
Therefore, if $y\in L_1+\mathbb{Z}^3$, then
\[
d(y, L_1)+c_1=c_1\leq \liminf_{\ep\to 0}v(y, \ep e_1).
\]
We hence only need to consider the case $y\notin L_1+\mathbb{Z}^3$.   
Fix $0< \delta< \frac{1}{8}{\rm dist}(y, L_1+\mathbb{Z}^3)$. For $\ep>0$, choose  $\xi_{\ep}\in \Gamma_{i,t_{\ep},y}$ for some  $t_{\ep}>0$ such that 
\[
v(y, \ep e_1)\geq \int_{0}^{t_\ep}\left({1\over 2}|\dot \xi_{\ep}|^2-\dot \xi_{\ep}\cdot \ep e_1-V(\xi_{\ep})+\ol H(\ep e_1)\right)\,ds+c_{1\ep}-\ep.
\]
Note that for $t_1, t_2\in [0,t_{\ep}]$ with $t_1<t_2$,
\begin{align}\label{main-control}
&\int_{t_1}^{t_2}\left({1\over 2}|\dot \xi_{\ep}|^2-\dot \xi_{\ep}\cdot \ep e_1 -V(\xi_{\ep})+\ol H(\ep e_1)\right)\,ds\\
\geq  \ & \frac{t_2-t_1}{2}\left |{\xi_{\ep}(t_2)-\xi_{\ep}(t_1)\over t_2-t_1}-\ep e_1 \right |^2-\int_{t_1}^{t_2}V(\xi_{\ep})\,ds \notag
\end{align}
For $1 \leq i \leq 3$, denote by $L_{i,\del}$   the cylinder around $L_i$ with radius $\del$, that is,
\[
L_{i,\del}=\left\{x\in \mathbb{R}^3 \,:\, {\rm dist}(x, L_i)\leq \del \right\}.
\]
Since $v(y,\ep e_1)$ is uniformly bounded for all $\ep\in [0,1]$, due to \eqref{main-control},  for our fixed $\delta$, we have that 

\begin{itemize}
\item the total amount of time that $\xi_\ep$ stays outside of the tubes $U=\bigcup_{i=1}^3 (L_{i,\del}+\mathbb{Z}^3)$ is bounded by a constant $T_{\delta}>0$, which is independent of $\ep$;

\item  the number of times that $\xi_{\ep}$  jumps between different tubes in $U$, and the total jumping distance are bounded by a constant $N_{\delta}>0$, which is independent of $\ep$.   
Here, a jump of $\xi_\ep$ refers to an interval $[t_1,t_2] \subset [0,t_\ep]$ such that $\xi_\ep(t_1)$ and $\xi_\ep(t_2)$ belong to two different tubes in $U$, and $ \xi_{\ep}((t_1,t_2))\cap U = \emptyset$.

\end{itemize}

\noindent This situation is in fact quite similar to the Hedlund example. 
Now, denote by $\bar t_{\ep,\delta}$  the exit time of $\xi_{\ep}$ from $L_{1,\del}+\mathbb{Z}^3$, that is,
\[
\xi_{\ep}(\bar t_{\ep,\delta})\in (L_{1,\del}+\mathbb{Z}^3)  \quad \mathrm{and} \quad \xi_{\ep}((\bar t_{\ep,\delta},t_\ep])\cap (L_{1,\del}+\mathbb{Z}^3)=\emptyset.
\]
Thanks to the observations above, we have that
\[
|e_1\cdot (\xi_{\ep}(\bar t_{\ep,\delta})-y)|\leq C_{\delta},
\]
for a constant $C_{\delta}$ independent of $\ep$. Then,
\begin{align*}
v(y, \ep e_1)&\geq  \int_{\bar t_{\ep,\delta}}^{t_{\ep}}\left({1\over 2}|\dot \xi_{\ep}|^2-\dot \xi_{\ep}\cdot \ep e_1-V(\xi_{\ep})+\ol H(\ep e_1)\right)\,ds+c_{1\ep}-\ep\\
&\geq d(y,L_1)-C\delta- C_{\delta}\ep+c_{1\ep}-\ep.
\end{align*}
By first sending $\ep \to 0$, and then $\delta\to 0$, we have that
\[
\liminf_{\ep\to 0}v(y, \ep e_1)\geq d(y,L_1)+c_1.
\]
Therefore, our claim (\ref{non-cont2}) holds. 

It is clear that $d(y,L_i)=0$ for $y\in L_i+\Z^3$, and $d(y,L_i)>0$ for $y\in \T^3 \setminus (L_i+\Z^3)$. 
Combining \eqref{non-cont1} and \eqref{non-cont2} to get
\[
d(y,L_1) + c_1 = d(y,L_2) + c_2 = d(y,L_3) + c_3 \quad \text{ for all } y \in \T^3.
\]
Without loss of generality, we assume that $c_1\geq c_2$. 
Then, for $y\in L_2$,
\[
c_2=d(y,L_2)+c_2=d(y,L_1) + c_1 > c_1,
\]
which is a contradiction.
\end{proof}

\begin{rem}
Note that in this example we pointed out here, $\{V=0\}=L_1 \cup L_2 \cup L_3$ is not compactly supported in $\T^3$ and is of a non-trapping situation,
which is why we have the representation formula of $\ol{H}$ in Lemma \ref{lem:formula-Hbar}.

Because of a technical reason that we need $L_{1,1/8}, L_{2,1/8}, L_{3, 1/8}$ to be disjoint, we need to work in three dimensional space.
In fact, any dimension greater than or equal to three will work. We do not know yet how to  construct this kind of example in two dimensions. 

\end{rem}

\begin{rem} By refining our arguments,  we can actually show that there does not exist continuous selection of supersolutions to (E)$_p$ either. 
Indeed, assume otherwise that there is a way to select $w(y,p) \in \Lip(\T^3)$, a supersolution to (E)$_p$, such that $p \mapsto w(y,p)$ is continuous at $p=0$. Denote by
\[
c_{1 \ep} = \min_{y \in L_1} w(y,\ep e_1) = w(y_\ep, \ep e_1) \quad \text{ for some } y_\ep \in L_1.
\]
By passing to a subsequence if necessary, we assume further that $\lim_{\ep \to 0} c_{1 \ep } = c_1$, and $\lim_{\ep \to 0} y_{\ep} = y_0 \in L_1$.
By the comparison principle on Aubry set $L_1$, we have $w(y,\ep e_1) \geq v(y,\ep e_1)$, where $v(y, \ep e_1)$ is given in formula  \eqref{HL-Mather} with $c=c_{1\ep}$.
Let $\ep \to 0$ to yield
\[
w(y,0) \geq d(y,L_1) + c_1 \quad \text{ and } \quad w(y_0,0) = c_1.
\]
In particular, 
\[
\min_{y \in \T^3} w(y,0) = \min_{y \in L_1} w(y,0)= c_1 \quad \text{ and } \quad w(y,0) > c_1 \quad \text{ for } y \in L_2 \cup L_3.
\]
Repeat the above argument for $c_{2 \ep} = \min_{y \in L_2} w(y, \ep e_2)$ to yield the contradiction.

Hence, the method in \cite{CDI} cannot be employed to obtain an improved upper bound of $u^{\ep}-u$. 

\end{rem}


\section{Appendix}\label{appen}

\subsection{Circle homeomorphism}
We first present a result on circle homeomorphism.  A continuous function  $f:\R\to \R$ is called a {\it circle  homeomorphism}  if $f$ is strictly increasing and 
\[
f(x+1)=f(x)+1 \quad \text{ for all } x\in \R.
\]
Then, it is well-known that the Poincar\'e rotation number 
\[
\beta_f=\lim_{i\to \infty}{f^{i}(x)\over i}.
\]
exists and is independent of $x\in \R$. Moreover, 
\begin{equation}\label{circle-control}
|f^{i}(x)-f(x)-i\beta_f|\leq 1 \quad \text{for all $i\in \Z$}.
\end{equation}
Also,  $\beta_f={p\over q}\in \mathbb{Q}$ with $p\in \Z, q \in \N$ if and only if there exists $x_0\in \R$ such that 
\[
f^{q}(x_0)=f(x_0)+p.
\]
Here, for $i\in \N$, $f^{i}$ represents the $i$-th iteration of $f$. 
See \cite[Chapter 3, \textsection 11]{Ar} for details.

\subsection{Rotation vectors in two dimensions} 
Assume $n=2$. Let $H$ be a smooth Hamiltonian satisfying (H1)--(H3) and furthermore that $p\mapsto H(y,p)$ is strictly convex for each $y\in \T^2$.
Let  $L$ be the corresponding Lagrangian. 

A Lipschitz continuous curve $\gamma:\R\to \R^2$ is  called an {\it absolute minimizer} associated with $L+c$ for some $c\in  \R$ if 
\[
\int_{t_1}^{t_2}L(\gamma, \dot \gamma)+c\,ds\leq \int_{s_1}^{s_2}L(\del, \dot \del)+c\,ds
\]
for any $t_1<t_2$, $s_1<s_2$ and  Lipschitz continuous curve $\del$ satisfying $\del(s_i)=\gamma(t_i)$ for $i=1,2$. 
Two absolute minimizers associated with $L+c$ cannot intersect twice unless they are the same after suitable translation in time. 
This non-crossing property, together with the two dimensional topology, plays a crucial role in the Aubry-Mather theory, 
which provides detailed information about distributions of absolute minimizers (see \cite{Ba1}).
For reader's convenience, let us give a brief explanation about this important non-crossing fact here. 
Assume otherwise that two distinct (up to translation) absolute minimizers $\gam_1$ and $\gam_2$ intersect at least twice.
By translation, we may assume that there are $a, b_1, b_2 \in \R$ such that $a<b_1 \leq b_2$ and
\[
\gam_1(a) = \gam_2(a), \quad \gam_1(b_1) = \gam_2(b_2).
\]
It is clear that
\[
\int_{a}^{b_1}L(\gamma_1, \dot \gamma_1)+c\,ds =  \int_{a}^{b_2}L(\gamma_2, \dot \gamma_2)+c\,ds.
\]
Let $\gam_3:[a,b_1+1] \to \R^2$ be such that
\[
\gam_3(s)=
\begin{cases}
\gam_1(s) \quad &\text{ for } s \in [a,b_1],\\
\gam_2(s+b_2-b_1) \quad &\text{ for } s \in [b_1, b_1+1].
\end{cases}
\]
Since 
\[
\int_{a}^{b_2+1}L(\gamma_2, \dot \gamma_2)+c\,ds =  \int_{a}^{b_1+1}L(\gamma_3, \dot \gamma_3)+c\,ds,
\]
  $\gam_3|_{[a,b_1+1]}$ is also a minimizer of the action
\[
\int L(\gamma , \dot \gamma )+c\,ds 
\]
with corresponding fixed endpoints $\gamma_2(a)=\gamma_3(a)$ and $\gamma_2(b_2+1)=\gamma_3(b_1+1)$.
Hence $\gam_3$ is $C^2$ and solves the following Euler-Lagrange equations
\[
\frac{d}{ds} \left( D_q L(\gam_3(s),\dot \gam_3(s)) \right) = D_x L(\gam_3(s), \dot \gam_3(s)) \quad \text{ for all } s \in [a, b_1+1].
\]
Accordingly, at the junction $\gamma_1(b_1)=\gamma_2(b_2)=\gamma_3(b_1)$, we must have that
\[
\gamma_{1}'(b_1)=\gamma_{2}'(b_2)=\gamma_{3}'(b_1).
\]
Since $\gamma_1$ and $\gamma_2$ are also solutions to the above Euler-Lagrange equation,  the uniqueness result of second order ODEs yields that $\gamma_1(t)=\gamma_2(t+b_2-b_1)$, which is absurd.

\smallskip

The following theorem was established in \cite{Car}. 
\begin{thm}\label{thm:2D-level}
If $\ol H(p)>\min_{\R^2}\ol H$, the set of sub-differentials $\partial \ol H(p)$ is a radical segment, that is,
\[
\partial \overline{H}(p) = \left\{r n_p \,:\, r \in [r_1(p),r_2(p)]\right\}  \quad \text{for some unit vector  $n_p$ and $r_1(p), r_2(p)>0$.}
\]
In particular, this implies that for $s>\min_{\R^2} \overline{H}$, the level curve $\{\overline H=s\}$ is $C^1$.
\end{thm}

Note that the above is in general false when $n\geq 3$ (see Lemma \ref{lem:formula-Hbar}). For readers' convenience, we provide a sketch of the proof. 
The basic idea is quite simple. 
Suppose that $q_1,q_2\in \partial \ol H(p)$. 
Then, there exist two Mather measures $\mu_1$ and $\mu_2$ associated with $p$ such that, for $i=1,2$, the rotation vector
\[
\iint_{\T^2\times \R^2}q\,d\mu_i=q_i.
\]
Accordingly, if $q_1$ and $q_2$ are not parallel,  then there are two different orbits from supports of $\mu_1$ and $\mu_2$ respectively,  which intersect each other. 
This is impossible. 

This conclusion also implies the existence of $s_p \in [-\infty,\infty]$ such that,
for any orbit $\xi(t)=(x(t),y(t))$ on the Mather set $\widetilde{\mathcal{M}}_p$,
\[
\lim_{|t| \to \infty} \frac{y(t)}{x(t)} = s_p.
\]

\subsection{Identification with circle homeomorphisms}

Fix $p\in \R^2$. Let $\widetilde{\mathcal{M}}_p$ be the Mather set corresponding to $p$.

Assume that  $\xi:\R \mapsto \T^2$  is an orbit on $\widetilde{\mathcal{M}}_p$, i.e., 
\[
\left\{(\xi(t),\dot \xi(t))\,:\,t\in  \R \right\} \subset \widetilde{\mathcal{M}}_p.
\]
We lift $\xi$ to $\R^2$ and still write it the same by abuse of notation.  
Let us note right away that orbits on $\widetilde{\mathcal{M}}_p$ do not intersect with each other since they are all absolute minimizers of $\int L(\gam,\dot \gam)+\ol H(p)\,ds$.  Therefore, they are totally ordered in $\R^2$ (see Figure \ref{fig:orbit}).

In the following, we explain how to associate $\xi=(x(t),y(t))$ with a circle map $f:\R\to \R$ when $\ol H(p)>\min_{\R^2} \ol H$, which is well-known in the Aubry-Mather theory. See \cite[Chapter 3 and 6]{Ba1}  for details from geometric point of view. Without loss of generality, we assume that 
\[
\lim_{|t| \to \infty} \frac{y(t)}{x(t)} = s_p\in [0,1].
\]
Let  $\eta:\R \mapsto \T^2$ be a periodic  trajectory associated with  another Mather set on the same energy level  with rotation vector parallel to $(0,1)$.
More precisely, there exists $p'\in \R^2$ and $T>0$ such that   $\ol H(p')=\ol H(p)$,
\[
\left\{(\eta(t),\dot \eta(t))\,:\,t\in  \R \right\} \subset \widetilde{\mathcal{M}}_{p'},
\]
and 
\[
\eta(T)-\eta(0)=(0,1).
\]
The existence of periodic orbits $\eta$ with rational rotation vector is well-known (see \cite{Ba1}).   
Clearly, for each $k\in \Z$,  $\xi$ intersects with $\eta_k=\eta+(k,0)$ exactly once since both $\xi$ and $\eta$ are absolute minimizers of the action 
\[
\int L(\gamma, \dot \gamma)+\ol H(p)\,ds.
\]
Without loss of generality, we may assume that 
\[
\xi(0)=\eta(0)\in  [0,1]^2.
\] 
For each $k\in \Z$, let $a_k\in \R$ be such that 
\[
\xi\cap \eta_k=\eta_k(a_kT).
\]
Since  orbits on $\widetilde{\mathcal{M}}_p$ are totally ordered in $\R^2$,
either $a_k=0$ for all $k\in \Z$ or $\{a_k\}_{k\in \Z}$ is a strictly increasing sequence.  
Moreover, for fixed $k>l$, 
\[
a_k-a_l=i\in \Z  \quad \Longrightarrow \quad a_{k+m}-a_{l+m}=i \quad \text{for all } m\in \Z.
\]
Indeed, $a_k-a_l=i$ means that there exists $\alpha \in \R$ such that $\xi(s+\alpha) = \xi(s)+(k-l,i)$ for all $s\in \R$, and hence, the implication follows.

Thus, there exists a circle homeomorphism $f$ such that 
\[
f(a_k)=a_{k+1}  \quad \text{for all $k\in \Z$}.
\]
See \cite[Theorem 3.15]{Ba1} for further details on the definition of $f$.

\begin{figure}[h]
\begin{center}
\begin{tikzpicture}

\draw plot [smooth] coordinates {(-2,0.1) (-1,0.5) (0,1.1) (1,1.5) (2,1.9) (3,2.5) (4,2.8)};
\draw (-2,0.3) node {$\tilde \xi$};
\draw[->] (-1,0.5)--(-0.95,0.51);

\draw plot [smooth] coordinates {(-2,-1.9) (-1,-1.5) (0,-0.9) (1,-0.5) (2,-0.1) (3,0.5) (4,0.8)};
\draw (-2,-2.3) node {$\xi$};
\draw[->] (-1,-1.5)--(-0.95,-1.49);

\draw plot [smooth] coordinates {(-2,-1) (-1,-0.5) (0,-0.3) (1,0.3) (2,1) (3,1.4) (4,2)};
\draw plot [smooth] coordinates {(-2,-0.8) (-1,-0.3) (0,-0.1) (1,0.5) (2,1.3) (3,1.7) (4,2.2)};

\draw[red, thick, dashed] plot [smooth] coordinates {(0,-3) (0.3,-2) (0,-1) (0.3,0) (0,1) (0.3,2)};
\draw[red, thick, dashed] plot [smooth] coordinates {(1,-3) (1.3,-2) (1,-1) (1.3,0) (1,1) (1.3,2)};
\draw[red, thick, dashed] plot [smooth] coordinates {(2,-3) (2.3,-2) (2,-1) (2.3,0) (2,1) (2.3,2)};

\draw (0,-3.3) node {$\eta_0$};
\draw (1,-3.3) node {$\eta_1$};
\draw (2,-3.3) node {$\eta_2$};
\end{tikzpicture}
\caption{Orbits in $\widetilde{\cM}_p$ and $\{\eta_k\}_{k\in \Z}$} \label{fig:orbit}
\end{center}
\end{figure}
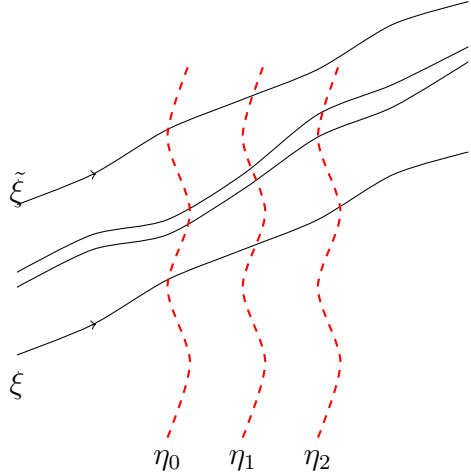

\subsection{Uniform convergence of  slope of orbits on Mather sets.}

Through suitable translations, we may assume that 
\[
\min_{\R^2}\ol H=\ol H(0)=0.
\]
Let $p$, $p'$, $\xi$ and $\eta$ be from the previous section. Denote $0<r=\ol H(p)=\ol H(p')$,
\[
S_r=\left\{\ol H=r\right\}, \quad a_r=\max_{\tilde p\in S_r}\left|\tilde p \right|, \quad    b_r=\max\left\{|D_pH(y,\tilde p)| \,:\,  H(y,\tilde p)=r\right\},
\]
and
\[
 \quad d=\max\left\{|\tilde p|: \  H(y,\tilde p)=0\right\}.
\]
Let $v$ and $v'$ be viscosity solutions to 
\[
H(y,p+Dv)=\ol H(p)=r \quad \text{ in } \T^2
\]
and
\[
H(y,p'+Dv')=\ol H(p')=r \quad \text{ in } \T^2,
\]
respectively.  Note that $v$, $v'$ are smooth along $\xi$, $\eta$, respectively, and 
\[
\dot \xi=D_pH(\xi, p+Dv(\xi))     \quad  \mathrm{and} \quad \dot \eta=D_pH(\eta, p'+Dv'(\eta)).
\]
Therefore,
\begin{equation}\label{speed-bound}
|\dot \xi|,  \  |\dot \eta|\leq b_r.
\end{equation}
Also,  for any Lipschitz continuous curve $\gamma:[s_1,s_2]\to \R^2$, 
\begin{align}\label{L-lowerbound}
\int_{s_1}^{s_2}L(\gamma, \dot \gamma)\,ds&= \int_{s_1}^{s_2}L(\gamma, \dot \gamma)+H(\gam, Dw)\,ds\\
&\geq \int_{s_1}^{s_2}\dot \gamma\cdot Dw\,ds
=w(\gamma(s_2))-w(\gamma(s_1))\geq -d \sqrt{2}. \notag
\end{align}
Here $w$ is a viscosity solution to (E)$_0$.
Although $w$ might not be differentiable along $\gamma$, the above calculation can be easily justified by using standard convolutions of $w$.

Next we give  uniform lower and upper bounds of the period $T$ of $\eta$.  Since the rotation vector
\[
{\eta(T)-\eta(0)\over T}=\frac{1}{T} (0,1)\in  \partial \ol H(p')
\]
and $p'\cdot q\geq \ol H(p')=r$ for any $q\in \partial \ol H(p')$, we have that 
\begin{equation}\label{period-bound}
{1\over b_r}\leq T\leq   {a_r\over r}.
\end{equation}

\begin{lem}\label{slope-rate} Assume that $s_p\in [0,1]$. Then for all $t\in  \R$, 
\[
\left|y(t)-s_px(t)\right|\leq {C}.
\]
Here $C$ is  a constant depending only on $r$, $a_r$, $b_r$ and $d$. 
\end{lem}

\begin{proof} 
Assume that 
\[
\xi(t_k)=\eta_k(a_k T).
\]
In light of \eqref{speed-bound} and \eqref{period-bound}
\[
\left|\eta(t)-\eta(0)-\left(0,{t\over T}\right)\right|\leq C.
\]
Therefore,
\[
|x(t_k)-k|\leq C  \quad \mathrm{and} \quad   \left|y(t_k)-a_k\right|\leq C.
\]
Hence, $\lim_{k\to \infty} {a_k\over k}=s_p$.  Since $t_0=a_0=0$, thanks to   \eqref{circle-control},
\[
\left|a_k-k s_p \right|\leq 1.
\]
Accordingly, 
\begin{equation}\label{bound-k}
\left|y(t_k)-s_p x(t_k)\right|\leq {C} \quad \text{for all $k\in \Z$}.
\end{equation}
Next we claim that 
\begin{equation}\label{bound-t}
|t_k-t_{k+1}|\leq C  \quad \text{for all $k\in \Z$}.
\end{equation}
In fact, from the previous calculations, we have that 
\[
|x(t_{k+1})-x(t_k)|\leq C \quad \mathrm{and} \quad  |y(t_{k+1})-y(t_k)|\leq C,
\]
which, together with \eqref{L-lowerbound},  yields
\begin{align*}
C \geq p\cdot( \xi(t_{k+1})-\xi(t_k))+v(\xi (t_{k+1})-v(\xi(t_k))&=\int_{t_k}^{t_{k+1}}L(\xi, \dot \xi)+\ol H(p)\,ds\\
&= r(t_{k+1}-t_k)+\int_{t_k}^{t_{k+1}}L(\xi, \dot \xi)\,ds\\
&\geq r(t_{k+1}-t_k)-d \sqrt{2}.
\end{align*}
Thus, \eqref{bound-t} holds.  Combine \eqref{speed-bound}, \eqref{bound-k} and \eqref{bound-t} to imply
\begin{equation}\label{continuous-control}
\left|{y(t)}-s_p x(t)\right|\leq {C} \quad \text{for all $t\in \R$}.
\end{equation}

\begin{rem}\label{rem:constant-issue}  We would like to highlight that constant $C$ in \eqref{continuous-control} above does not depend on the smoothness of $H$. Write 
\begin{align*}
&M_r(x)=\max\left\{|p|\,:\, H(x,p)=r \right\},   \\ 
&D_r(x)=\min\left\{|p_2-p_1| \,:\,   H(x,p_2)=2r, \  H(x,p_1)=r \right\}.
\end{align*}
Then it is easy to see that 
\[ 
a_r\leq \max_{x\in \R^2}M_r(x)  \quad \mathrm{and} \quad b_r\leq {r\over \min_{x\in \R^2}D_r(x)}.
\]

\end{rem}

\end{proof}


\end{document}